\documentclass[pdflatex,sn-mathphys-num]{sn-jnl}

\usepackage{graphicx}%
\usepackage{multirow}%
\usepackage{amsmath,amssymb,amsfonts}%
\usepackage{amsthm}%
\usepackage{mathrsfs}%
\usepackage[title]{appendix}%
\usepackage{xcolor}%
\usepackage{textcomp}%
\usepackage{manyfoot}%
\usepackage{booktabs}%
\usepackage{algorithm}%
\usepackage{algorithmicx}%
\usepackage{algpseudocode}%
\usepackage{listings}%

\theoremstyle{thmstyleone}%
\newtheorem{theorem}{Theorem}
\newtheorem{proposition}[theorem]{Proposition}%

\theoremstyle{thmstyletwo}%
\newtheorem{remark}{Remark}%
\newtheorem{lemma}{Lemma}

\theoremstyle{thmstylethree}%
\newtheorem{definition}{Definition}%

\begin{document}

\title[Insensitizing controls for a quasi-linear parabolic equation with diffusion depending on gradient of the state]{Insensitizing controls for a quasi-linear parabolic equation with diffusion depending on gradient of the state}


\author[1]{\fnm{Dany} \sur{Nina Huaman}}\email{danynina3003@gmail.com}
\equalcont{These authors contributed equally to this work.}
\author*[2]{\fnm{Miguel R.} \sur{Nu\~{n}ez-Ch\'{a}vez}}\email{miguel.chavez@ufmt.br}
\equalcont{These authors contributed equally to this work.}

\affil[1]{\orgdiv{Departamento de Matem\'atica}, \orgname{Universidad Nacional Agraria La Molina}, \orgaddress{\city{La Molina}, \state{Lima}, \country{Per\'u}}}

\affil*[2]{\orgdiv{Departamento de Matem\'atica}, \orgname{Universidade Federal de Mato Grosso}, \orgaddress{\city{Cuiab\'a}, \state{Mato Grosso}, \country{Brasil}}}


\abstract{In this paper, a quasi-linear parabolic equation with a diffusion term dependent on the gradient to the state with Dirichlet boundary conditions is considered. The goal of this paper is to prove the existence of control that insensitizes the system under study which is the case that Xu Liu left open in 2012. It is well known that the insensitizing  control problem is equivalent to a null controllability result for a cascade system, which  is obtained by duality arguments, Carleman estimates, and the Right Inverse mapping theorem. Also, some possible extensions and open problems concerning other quasi-linear systems are presented.}




\maketitle


\section{Introduction}\label{sec1}
 The problem of insensitizing  control was originally address by J. L. Lions in \cite{Lion1, Lion2}, leading to numerous papers on this topic  both in hyperbolic and parabolic equations.

Concerning the semi-linear heat equation, the first result was obtained in \cite{Bodart1} for a distributed control, more precisely
\begin{equation}\label{I-EC1}
\left\{\begin{array}{lll}
  y_t -   \Delta y+f(y)= \xi+u\chi_{\omega}  & \text{in}& \ \Omega\times(0,T),\\
  y= 0 &  \text{on} &\ \partial\Omega\times(0,T),\\
  y(0) = y_{0}+\tau \hat{y}_{0}&  \text{in} &\ \Omega,\\
\end{array}\right.
\end{equation}
where  $\Omega$  is open  set in  $\mathbb{R}^{N}$; $\chi_{\omega}$  denotes  the characteristic  function  of  the  open set  $\omega,\, \omega\subset\Omega;\, \xi,\, \hat{y}_{0}$  are  given  in $X_{1},\, X_{0}$ (certain spaces). The  data  of  the  state  equation  \eqref{I-EC1} is  incomplete  in  the  following  sense:
\begin{itemize}
	\item 
	$\hat{y}_{0}\in X_{0}$  is unknown  and  $\|\hat{y}_{0}\|_{X_{0}}=1$. 
	
	\item 
	$\tau\in \mathbb{R}$  is  unknown and  small  enough.
\end{itemize}
Also, considering  the  energy functional:
\begin{equation}\label{eq1}
\Phi(y):= \dfrac{1}{2} \iint_{\mathcal{O} \times (0,T)} |y(x,t;\tau,u)|^2 \,dx\,dt,
\end{equation}
the problem  of  insensitizing  controls   can  be   stated,  roughly, as follows: 
 
 \begin{itemize}
	 \item 
	We  say  that  the  control  $u$  insensitizes  $\Phi$ if
\begin{equation}\label{eq-1-1}
	\left|\frac{\partial \Phi}{\partial \tau}(y(\cdot,\cdot;\tau,u)) \Big|_{\tau=0} \right|=0,
\end{equation}
when $\eqref{eq-1-1}$  holds  the  functionals $\Phi$  is  locally  insensitive  to  the perturbation $\tau \hat{y}^{0}$.
    \item 
	Given $\epsilon>0,$  the  control  $u$  is said  to  $\epsilon$-insensitize $\Phi$ if
\begin{equation}\label{eq-1-2}
	\left|\frac{\partial \Phi}{\partial \tau}(y(\cdot,\cdot;\tau,u)) \Big|_{\tau=0} \right|\leq \epsilon.
\end{equation}

 \end{itemize}

In \cite{Bodart1}, the authors introduced and studied the notion of approximate insensitizing controls ($\epsilon$-insensitizing controls) of  \eqref{I-EC1}.  In order to get rid of the condition $y_0=0$, the authors  prove that  the problem of  $\epsilon$-insensitizing controls  is equivalent to an approximate controllability result for a cascade system which is established  therein. 

In \cite{Tere-z}, the condition $y_0=0$ was removed for the linear heat equation, instead the following condition was imposed: $\mathcal{O}\subset \omega$ or $\mathcal{O}=\Omega$. Furthermore, the authors proved that if the imposed condition is not satisfied, some negative results occur. In \cite{Tere1}, the author proved the existence of insensitizing controls for the same semilinear heat system and proved the existence of an initial datum in $L^2$ that cannot be insensitized. This last result is extended in \cite{Bodart2} to super-linear nonlinearities.

For parabolic systems arising from fluids dynamics the first attempt to treat the insensitizing problem is \cite{Cara3} for a large scale ocean circulation model (linear). In \cite{Guerrero2}, as we have already mentioned, the author treats both the case of a sentinel given by $L^2$-norm of the state and $L^2$-norm of the curl of the state of a linear Stokes system.
As long as insensitizing controls have been considered the condition $\omega\cap\mathcal{O}\neq \emptyset$ has always been imposed. But, from  \cite{TERE2} and \cite{Micu1}, we see that this is not a necessary condition for $\epsilon~$-~insensitizing controls. For instance, the authors have proved in \cite{Micu1} that there exist $\epsilon$-insensitizing controls of $\Phi$ for linear heat equations with no intersecting observation and control regions in one space dimension using the spectral theory.
Furthermore, the insensitizing problem, as we have seen in this special case, is directly related to control problems for coupled systems. In particular, one could ask whether it is possible to control both states of a coupled system just  by acting on one equation. In \cite{Guerrero2} and \cite{Gue1}, as well as some insensitizing problems, the author studied this problem respectively for Stokes and heat systems in a more general framework.

Continuing, in \cite{Simpo} the authors studied the insensitizing control problem with constraints on the control for a nonlinear heat equation for more general cost functionals, by means of the Kakutani's fixed point theorem combined with an adapted Carleman inequality. Some controllability results for equations  of the quasilinear parabolic type have been studied in  \cite{Dany-s, Irene-j,  Mig}, with respect to the   insensitizing  controllability   for  quasi-linear  parabolic  equations  have  been studied by Xu Liu in \cite{Xu-1} (2012). The author worked with H\"older space for the state $y$ and control $u$, more  precisely  the following  system was studied
\begin{equation}\label{I-EC2}
\left\{\begin{array}{lll}
  \displaystyle y_t -   \sum_{i,j=1}^{N} \left(a^{ij}(y)y_{x_{i}}\right)_{x_{j}}+f(y)= \xi+u\chi_{\omega}  & \text{in}& \ \Omega\times(0,T),\\
  y= 0 &  \text{on} &\ \partial\Omega\times(0,T),\\
  y(0) = y_{0}+\tau \hat{y}_{0}&  \text{in} &\ \Omega,\\
\end{array}\right.
\end{equation}
with the energy functional \eqref{eq1}.

When  the diffusion coefficient depend of gradient of state ($a^{ij}(\nabla y)$), this case was left open by Xu Liu  in \cite{Xu-1}. In this paper the main novelty is that we get to proof the existence of  insensitizing controls for the system \eqref{I-EC2} with the diffusion coefficient dependent on gradient of state. 


\section{Statements of the Main Results}\label{sec2}
\setcounter{equation}{0}

Let $\Omega \subset \mathbb{R}^{N}(N=1, 2 \  \mbox{or} \  3)$ be a bounded domain whose boundary $\Gamma$ is regular enough. Let $T > 0$ be given and let us consider the cylinder $Q=\Omega \times (0,T)$, with lateral boundary $\Sigma = \Gamma \times (0,T)$. Assume $\omega$ and $\mathcal{O}$ to be two given non-empty open subsets of $\Omega$. In the sequel, we will denote by $\chi_{\omega}$ the characteristic function of the subset $\omega$; $(\cdot,\cdot)$ and $\|\cdot\|$  respectively the $L^2$ scalar product and the norm in $\Omega$. The symbol $C$ is used to design a generic positive constant. 

We will consider the following system
\begin{equation}\label{EC1}
\left\{\begin{array}{lll}
  y_t -  \nabla \cdot\left(a(\nabla y) \nabla y\right)+f(y)= \xi+u\chi_{\omega}  & \text{in}& \ Q,\\
  y= 0 &  \text{on} &\ \Sigma,\\
  y(0) = y_{0}+\tau \hat{y}_{0}&  \text{in} &\ \Omega.\\
\end{array}\right.
\end{equation}

In system \eqref{EC1}, $y$ and $u$ are  respectively the  state variable and the control variable, $\xi$ and $y_{0}$ are two known functions,  $\tau$ is an unknown small real number, $\hat{y}_{0}$ is an  unknown function, $f(\cdot)$ and $a(\cdot)$ are given functions in $C^{2}(\mathbb{R};\mathbb{R})$ and $C^{3}(\mathbb{R}^{N};\mathbb{R})$ respectively, satisfying

\begin{equation}\label{hipotesis}
\left\{\begin{array}{l}
	0<a_{0}\leq a(x)\leq a_{1},\,\,\,\forall \, x\in \mathbb{R}^{N}\\
\noalign{\smallskip}

\|D_{i}a(x)\|_{\mathbb{R}^{N}} + \|D^2_{ij}a(x)\|_{\mathbb{R}^{N^{2}}} + \|D^3_{ijk}a(x)\|_{\mathbb{R}^{N^3}}\leq M, \ \forall \ x\in \mathbb{R}^{N},\,\, \\
\noalign{\smallskip}

f(0)=0,\,\,\ |f(r)|+|f'(r)|+|f''(r)| \leq M,  \,  \, \forall r\in \mathbb{R},
\end{array}\right.\\ 
\end{equation}
where $D_i a$, $D^2_{ij} a$ and $D^3_{ijk} a$ denote the first, second and third order derivatives (respectively) of the function $a(\cdot)$.

Also, we will consider the energy functional $\Phi(\cdot)$ for the system \eqref{EC1} as in \eqref{eq1}.

Let us define the following spaces
$$
X_{0}= \Big\{ w:\, w\in L^{2}(0,T;H^{2}(\Omega)), w_{t}\in L^{2}(0,T;L^{2}(\Omega)), w|_{\Sigma}=0 \Big\},
$$
and
$$
X_{1}=\Big\{ w:\, w\in L^{2}(0,T;H^{4}(\Omega)),\, w_{t}\in L^{2}(0,T;H^{2}(\Omega)),\, w_{tt}\in L^{2}(0,T;L^{2}(\Omega)),\, w|_{\Sigma}=0 \Big\}.
$$
Then, if $f=f(\cdot)$ and $a=a(\cdot)$ satisfy \eqref{hipotesis}, $\xi,\, u\chi_{\omega}\in X_{0}$, $y_{0},\, \hat{y}_{0}\in H^{3}(\Omega)\cap H^{1}_{0}(\Omega)$ with $\Delta y_0, \Delta \hat{y}_0 \in H_0^1(\Omega)$ and the functions $\xi,\, u\chi_{\omega},\, y_{0},\, \tau $ are sufficiently small in their spaces respectively, then    the equation \eqref{EC1} admits a unique solution  $y(\cdot,\cdot;\tau,u)\in X_{1}$ satisfying
\begin{equation}\label{H}
    \|y\|_{X_{1}}\leq C\left(\|\xi\|_{X_{0}} + \|u\chi_{\omega}\|_{X_{0}} + \|y_{0}+\tau \hat{y}_{0}\|_{H^{3}(\Omega)}\right),
\end{equation}
where $C$ is a positive constant depending only on $N,\, \Omega,\, T,\, M,\, a_{0}$ and $a_{1}$.\\
A result concerning the well possessedness of system \eqref{EC1} can be  found in the Appendix.\\

Now, we introduce the following notion:

\begin{definition}\label{def-ins}
For a given function $\xi\in X_{0}$ and $y_{0}\in H^{3}(\Omega)\cap H^{1}_{0}(\Omega)$, a control function $u \in X_{0}$ with $\mathrm{supp}\ u\subseteq \omega \times [0,T]$ is said  to insensitize the functional $\Phi$ defined in \eqref{eq1} if $u$ satisfies
\begin{equation}\label{eq-def-ins}
\frac{\partial \Phi}{\partial \tau}(y(\cdot,\cdot;\tau,u)) \Big|_{\tau=0}=0,\,\, \forall \hat{y}_{0}\in H^{3}(\Omega)\cap H^{1}_{0}(\Omega)\,\, \mbox{ with }\,\, \|\hat{y}_{0}\|_{H^{3}(\Omega)}=1.
\end{equation}
\end{definition}
Our main result is stated in the following Theorem:
\begin{theorem}\label{main-t}
Assume that $\omega\cap \mathcal{O}\neq \emptyset$ and $y_{0}=0.$ Then, there exist two positive constants $\tilde{M}$ and $\delta$ depending only on $N,\, \Omega,\, T,\, M,\, a_{0}$ and $a_{1}$, such that for any $\xi\in X_{0}$ satisfying
\begin{equation}\label{cond-xi}
	\left\|e^{{\frac{\tilde{M}}{t}}}\ \xi \right\|_{X_{0}} \leq \delta,
\end{equation}
one can find a control function $u\in X_{0}$  with $\mathrm{supp}\ u\subset \omega\times[0,T]$, which insensitizes the functional $\Phi$ in the sense of Definition \ref{def-ins}.
\end{theorem}

The rest of the paper is organized as follows. 

In Section 3, we will prove some technical results, this is, firstly, we will reformulate the insensitizing problem to null controllability problem for a optimality quasilinear system; secondly, we will study null controllability of the linearized system (to the optimality quasilinear system) applying Carleman estimates and finally we will find some additional estimates for the state.  

In Section 4, we will give the proof  of  Theorem~$\ref{main-t}$, thus, we will prove the null controllability for the optimality quasi-linear system applying the Liusternik's method.

Section 5 will deal with some additional comments and results. 

In Appendix, we will give a proof of the  well-possessedness of the quasilinear system~\eqref{EC1}.


\section{Preliminary results}\label{sec3}
\setcounter{equation}{0}

\subsection{Reformulation of the insensitizing problem:}

The special form of $\Phi$ allows us to reformulate  our insensitizing  problem as a controllability problem of a cascade system (for more details, see \cite{Xu-1}, for instance).
\begin{proposition}\label{char-e}
Assume that  $\xi, \, u \in X_{0}$ satisfies \eqref{cond-xi} and $y_{0}=0$. If $u$ and $\xi$  are sufficiently small in their spaces respectively    and the corresponding solution $(y,h)\in {X_{1} \times X_{0}}$ of the following nonlinear system
\begin{equation}\label{eq-char}
\left\{\begin{array}{lll}
y_{t}-\nabla\cdot (a(\nabla y)\nabla y)+f(y)=\xi+u\chi_{\omega} & \mbox{in} &Q,\\
\noalign{\smallskip}
-h_{t}-\nabla \cdot \left\{Da(\nabla y)(\nabla y \cdot \nabla h)+a(\nabla y)\nabla h\right\} +f'(y)h=y\chi_{\mathcal{O}} &\mbox{in} & Q,\\
\noalign{\smallskip}
y=0,\, h=0  &\mbox{on} & \Sigma,\\
\noalign{\smallskip}
y(x,0)=0,\, h(x,T)=0 &\mbox{in} & \Omega,
\end{array}\right.
\end{equation}
satisfies $h(x,0)=0$ in $\Omega,$ then $u$ insensitizes the functional $\Phi$ $($defined by \eqref{eq1}$)$.
\end{proposition}
\begin{proof}
Proof of this Proposition  can be  found in the Appendix \ref{Apex_3}.
\end{proof}

In order to prove that the system \eqref{eq-char} is null controllable at time $t=0$, we consider the linearized system for \eqref{eq-char}

\begin{equation}\label{linear-1}
\left\{\begin{array}{lll}
y_{t}-a(0)\Delta y+Ay=u\chi_{\omega}+g_{1} & \mbox{in} &Q,\\
\noalign{\smallskip}
-h_{t}-a(0) \Delta h+Ah=y\chi_{\mathcal{O}}+g_{2} &\mbox{in} & Q,\\
\noalign{\smallskip}
y=0,\, h=0  &\mbox{on} & \Sigma,\\
\noalign{\smallskip}
y(x,0)=0,\, h(x,T)=0 &\mbox{in} & \Omega,
\end{array}\right.
\end{equation}
where  $A=f'(0)$ and the adjoint system of \eqref{linear-1}
\begin{equation}\label{adjoint-1}
\left\{\begin{array}{lll}
-\varphi_{t}-a(0)\Delta \varphi+A\varphi=\psi \chi_{\mathcal{O}}+G_{1} & \mbox{in} &Q,\\
\noalign{\smallskip}
\psi_{t}-a(0) \Delta \psi+A\psi=G_{2} &\mbox{in} & Q,\\
\noalign{\smallskip}
\varphi=0,\, \psi=0  &\mbox{on} & \Sigma,\\
\noalign{\smallskip}
\varphi(x,T)=0,\, \psi(x,0)=\psi_{0}(x) &\mbox{in} & \Omega.
\end{array}\right.
\end{equation}

\subsection{Carleman estimates for \eqref{adjoint-1} }

In the context of the null controllability  analysis of parabolic systems, Carleman estimates are a very
 powerful tool (see \cite{Furs, Ima}). In order to state our Carleman estimates we need  to define some weight functions. Let  $\omega_{0}$ be a non-empty open subset of $\omega \cap \mathcal{O},$ and set

\begin{equation}\label{weight-1}
\sigma(x,t)=\frac{e^{4\lambda \|\eta^{0}\|_{\infty}}-e^{\lambda(2\|\eta^{0}\|_{\infty}+\eta^{0}(x))}}{t(T-t)},\,\, \xi(x,t)=\frac{e^{\lambda(2\|\eta^{0}\|_{\infty}+\eta^{0}(x))}}{t(T-t)},
\end{equation}
for some parameter $\lambda>0.$ Here, $\eta^{0}\in C^{2}(\overline{\Omega})$ stands for a function that satisfies 
\begin{equation}\label{f-fursik}
|\nabla \eta^{0}|\geq k_{0}>0\,\, \mbox{ in}\,\, \Omega \setminus \overline{\omega}_{0},\, \eta^{0}>0\,\,  \mbox{in}\,\,  \Omega\,\, \mbox{ and} \,\, \eta^{0}=0\,\,\ \mbox{on} \,\,\, \partial \Omega.
\end{equation}
The proof of the existence of such a function $\eta^{0}$ can be found in \cite{Furs}.

This kind of weight functions was also used in \cite{Dany-s, Clark2}. We introduce the following notation:

$$
I(s,\lambda;\phi):=\iint_{Q}e^{-2s\sigma}\Big[(s\xi)^{-1}(|\phi_{t}|^{2}+|\Delta \phi|^{2})+\lambda^{2}s\xi |\nabla \phi|^{2}+\lambda^{4}(s\xi)^{3}|\phi|^{2}\Big]\, dxdt.
$$
We will deduce a Carleman inequality for the solutions to systems of the kind \eqref{adjoint-1}.

\begin{proposition}\label{carl-1}
Let us assume that $G_{1},\, G_{2}\in L^{2}(Q).$ There exist positive constants $\lambda_{1},\, s_{1}$  and $C_{1}$ such that, for any $s\geq s_{1}$ and $\lambda \geq \lambda_{1},$  any  $\psi_{0}\in L^{2}(\Omega),$ the corresponding solution to \eqref{adjoint-1}  satisfies

\begin{equation}\label{eq-carl-1}
\begin{array}{l}
\displaystyle I(s,\lambda;\varphi)+I(s,\lambda;\psi)\leq C_{1}\left(\iint_{Q}e^{-2s\sigma} \Big[\lambda^{4}(s\xi)^{3}|G_{1}|^{2}+|G_{2}|^{2}\Big]\, dxdt\right.\\
\noalign{\smallskip} 
	\phantom{ I(s,\lambda;\varphi)+I(s,\lambda;\psi)233Q}
\left. \displaystyle +\iint_{\omega\times (0,T)}e^{-2s\sigma}\lambda^{8}(s\xi)^{7}|\varphi|^{2}\, dxdt 
\right).
\end{array}
\end{equation}
Furthermore, $C_{1}$ and $\lambda_{1}$ only depend on $\Omega$ and $\omega$, and $s_{1}$ can be chosen of the form $s_{1}=C(T+T^{2})$, where $C$ only depends on $\Omega,\, \omega,\, a(0),\, \,\mbox{and}\,\,|A|$.
\end{proposition}

\begin{proof}
Obviously, it will be sufficient to show that the there exist $\lambda_{1}$ and $s_{1}$ such that, for any  small $\epsilon>0,$ any $s\geq s_{1}$ and any $\lambda \geq \lambda_{1},$ one has
\begin{equation}\label{eq-carl-p1}
\begin{array}{l}
\displaystyle I(s,\lambda;\varphi)+I(s,\lambda;\psi)\leq \epsilon I(s,\lambda;\psi)+C_{\epsilon}\left(\iint_{Q}e^{-2s\sigma} \Big[\lambda^{4}(s\xi)^{3}|G_{1}|^{2}+|G_{2}|^{2} \Big]\, dxdt\right.\\

\noalign{\smallskip}\phantom{QQQQQqqDDDDGG}
\displaystyle \left.+\iint_{\omega\times (0,T)}e^{-2s\sigma}\lambda^{8}(s\xi)^{7}|\varphi|^{2}\, dxdt
\right).
\end{array}
\end{equation}
From the usual Carleman inequalities (see \cite{Tere-B, Clark2}) in  \eqref{adjoint-1}, for  $s\geq \sigma_{1}(T+T^{2})$ and $\lambda\geq \lambda_{0},$ we have
\begin{equation}\label{eq-carl-p1-1}
\begin{array}{l}
\displaystyle I(s,\lambda;\varphi)+I(s,\lambda;\psi)\leq C\left(\iint_{Q}e^{-2s\sigma} \Big[|G_{1}|^{2}+|G_{2}|^{2} \Big]\, dxdt\right. \\
\noalign{\smallskip}\phantom{QQQQQqqDDDDGG}
\displaystyle \left.+\iint_{\omega_{0} \times(0,T)}e^{-2s\sigma}\lambda^{4}(s\xi)^{3} \Big[|\varphi|^{2}+|\psi|^{2} \Big]\, dxdt\right),
\end{array}
\end{equation}
from the last inequality we will obtain the inequality \eqref{eq-carl-p1}, this purpose we first let us introduce  a function  $\eta\in C^{\infty}_{0}(\omega)$ satisfying $0\leq \eta \leq 1$ and $\eta=1$ in $\omega_{0}.$ Then
$$
\begin{array}{l}
\displaystyle \iint_{\omega_{0}\times (0,T) }e^{-2s\sigma}\lambda^{4}(s\xi)^{3}|\psi|^{2}\, dxdt\leq \iint_{\omega \times (0,T)}e^{-2s\sigma}\lambda^{4}(s\xi)^{3}\eta |\psi|^{2}\, dxdt\\
\noalign{\smallskip}\phantom{QQQQQQQQQQQQQAAADD}
\displaystyle = \iint_{\omega \times(0,T)}e^{-2s\sigma}\lambda^{4}(s\xi)^{3}\eta \psi (-\varphi_{t}-a(0)\Delta \varphi +A \varphi -G_{1})\, dxdt\\
\noalign{\smallskip}\phantom{QQQQQQQQQQQQQAAADD} := M_{1}+M_{2}+M_{3}+M_{4}.
\end{array}
$$
Here, using the properties of function $\alpha$ and $\xi$. We get that
$$M_{i}\leq \epsilon  I(s,\lambda; \psi)+C_{\epsilon}\iint_{\omega \times (0,T)}e^{-2s\sigma}\lambda^{8}(s\xi)^{7}|\varphi|^{2}\, dxdt. $$
By these estimates and \eqref{eq-carl-p1-1}, we obtain the inequality \eqref{eq-carl-p1} and therefore this end the  proof.
\end{proof}

Now, we will prove a Carleman estimate with functions blowing up only  at $t=0;$ which  allows us to prove an observability inequality for the system \eqref{adjoint-1}.

Define the new weight functions
$$
\displaystyle \overline{\sigma}(x,t)=\frac{e^{\lambda \|\eta^{0}\|_{\infty}} - e^{\lambda(2\|\eta^{0}\|_{\infty} + \eta^{0}(x))}}{\ell(t)},\,\, \ \ \overline{\xi}(x,t)=\frac{e^{\lambda(2\|\eta^{0}\|_{\infty}+\eta^{0}(x))}}{\ell(t)},
$$
$$
\displaystyle \sigma^{\ast}(t)=\max_{x\in \overline{\Omega}}\overline{\sigma}(x,t),\ \ \ \  \xi^{\ast}(t)=\min_{x\in \overline{\Omega}}\overline{\xi}(x,t),
$$
$$
\displaystyle \hat{\sigma}(t)=\min_{x\in \overline{\Omega}}\overline{\sigma}(x,t),\ \ \ \ \hat{\xi}(t)=\max_{x\in \overline{\Omega}}\overline{\xi}(x,t),
$$
where the function $\ell$ is given by 
$$
\ell(t)= \begin{cases}
			t(T-t)\,\, &0\leq t\leq \dfrac{T}{2}, \\
\ \ \ \ \ \dfrac{T^{2}}{4}\,\, &\dfrac{T}{2} < t \leq T.
		 \end{cases}
$$
Notice that $\overline{\sigma}=\sigma$ and $\overline{\xi}=\xi$ in $\left(0,\frac{T}{2} \right)$ and let us introduce the notation
$$
\overline{I}(s,\lambda;\phi):=\iint_{Q}e^{-2s\overline{\sigma}} \Big[(s\overline{\xi})^{-1}(|\phi_{t}|^{2}+|\Delta \phi|^{2})+\lambda^{2}s\overline{\xi }|\nabla \phi|^{2}+\lambda^{4}(s\overline{\xi})^{3}|\phi|^{2} \Big]\, dxdt.
$$
One has the following:
\begin{proposition}\label{carl-2}
Let us assume that $G_{1},\, G_{2}\in L^{2}(Q).$ There exist positive constants $\lambda_{2},\, s_{2}$ such that, for any $s\geq s_{2}$ and $\lambda \geq \lambda_{2},$ there exists a  constant $C_{2}=C_{2}(s,\lambda)>0$ with  the following  property: for any $\psi_{0}\in L^{2}(\Omega),$  the associated solution to \eqref{adjoint-1} satisfies  
\begin{equation}\label{eq-carl-2}
\begin{array}{l}
\displaystyle \overline{I}(s,\lambda;\varphi)+\overline{I}(s,\lambda;\psi)\leq C_{2}\left(\iint_{Q}e^{-2s\overline{\sigma}} \Big[\lambda^{4}(s\overline{\xi})^{3}|G_{1}|^{2}+|G_{2}|^{2} \Big]\, dxdt\right.\\
\noalign{\smallskip} 
	\phantom{ I(s,\lambda;\varphi)+I(s,\lambda;\psi)233Q}
\left. \displaystyle +\iint_{\omega\times (0,T)}e^{-2s\overline{\sigma}}\lambda^{8}(s\overline{\xi})^{7}|\varphi|^{2}\, dxdt 
\right).
\end{array}
\end{equation}
Furthermore, $\lambda_{2}$ and $s_{2}$ only depend on $\Omega,\, \omega,\, T,\, a(0) \, \,\mbox{and}\,\,|A| $.
\end{proposition}
\begin{proof}
We can decompose all the integral in $\overline{I}(s,\lambda; \varphi)$  in the form $\displaystyle \iint_{Q}=\iint_{\Omega \times \left(0,\frac{T}{2} \right)}+\iint_{\Omega\times \left(\frac{T}{2},T \right)}$. 

Let us  gather together all the integrals in $\Omega \times \left(0,\frac{T}{2} \right)$\ (resp. $\Omega \times \left(\frac{T}{2},T \right)$) in $\overline{I}_{1}(s,\lambda; \varphi)$\  (resp. $\overline{I}_{2}(s,\lambda;\varphi)$). Then $\displaystyle \overline{I}(s,\lambda;\varphi)=\overline{I}_{1}(s,\lambda;\varphi)+\overline{I}_{2}(s,\lambda;\varphi)$ and a similar decomposition holds for $\overline{I}(s,\lambda;\psi).$

From Carleman inequality in Proposition \ref{carl-1}, with $s\geq s_{1}$ and $\lambda\geq \lambda_{1}$, we have
\begin{equation}\label{eta-0}
\begin{array}{l}
\overline{I}_{1}(s,\lambda;\varphi)+\overline{I}_{1}(s,\lambda;\psi)\leq C (\overline{RS}),
\end{array}
\end{equation}
where $(\overline{RS})$ is the right side in \eqref{eq-carl-2}.

In order to prove that $\displaystyle \overline{I}_{2}(s,\lambda;\varphi)+\overline{I}_{2}(s,\lambda;\psi)\leq C (\overline{RS}),$ let us define the function  $\eta \in C^{2}\left([0,T]\right)$ such that
$$
\eta(t)= \begin{cases}
			0,  &t \in \left[0,\dfrac{T}{4} \right],\\
			0 < \eta(t) < 1 &t \in \left( \dfrac{T}{4}, \dfrac{T}{2} \right),\\
			1,  &t \in \left[\dfrac{T}{2},T \right].
		 \end{cases}
$$
Then, if $(\varphi,\psi)$ is solution  of \eqref{adjoint-1}, it is not difficult  to see that $(\eta\varphi, \eta\psi)$ satisfies the system
\begin{equation}\label{eta-adjoint}
\left\{\begin{array}{lll}
-(\eta \varphi)_{t}-a(0)\Delta(\eta \varphi)+A(\eta\varphi)=(\eta \psi)\chi_{\mathcal{O}}+(\eta G_{1})-\eta_{t}\varphi & \mbox{in} &Q,\\
 	\noalign{\smallskip}
(\eta\psi)_{t}-a(0)\Delta (\eta \psi)+A(\eta \psi)=(\eta G_{2})+\eta_{t}\psi & \mbox{in} & Q,\\
 	\noalign{\smallskip}
\eta\varphi=0,\, \eta\psi=0  & \mbox{on}& \Sigma,\\
 	\noalign{\smallskip}
\eta \varphi(T)=0,\,\eta  \psi(0)=0 & \mbox{in}&\Omega.
\end{array}\right.
\end{equation}
The classical energy estimates for the equation $\eqref{eta-adjoint}_{2}$, we have
\begin{equation}\label{eta-1}
\int_{\Omega}|\eta \psi|^{2}\,dx+\iint_{Q}|\eta \nabla\psi|^{2}\, dxdt\leq C\left(\iint_{Q}|\eta G_{2}|^{2}\, dxdt+\iint_{Q}|\eta_{t}\psi|^{2}\, dxdt \right).
\end{equation}
Multiplying by $\eta \varphi$ in the first equation in \eqref{eta-adjoint}, integrating  in $Q$
$$
\iint_{Q}|\eta \varphi|^{2}\, dxdt+\iint_{Q}|\eta \nabla \varphi|^{2}\, dxdt\leq C \left(\iint_{Q}|\eta G_{1}|^{2}\, dxdt+\iint_{Q}|\eta_{t}\varphi|^{2}\, dxdt \right)+\epsilon \iint_{Q}|\eta \psi|^{2}\,dxdt.
$$
For $\epsilon>0$ sufficiently small and from this last inequality join with \eqref{eta-1}, we have
\begin{equation}\label{eta-2}
\begin{array}{l}
\displaystyle \iint_{Q}|\eta|^{2} \Big[|\varphi|^{2}+|\psi|^{2}+|\nabla\varphi|^{2}+|\nabla\psi|^{2}\Big]\, dxdt\\
 	\noalign{\smallskip}\phantom{QQWWWQ}
\displaystyle \leq C\left(\iint_{Q}|\eta|^{2}\Big[|G_{1}|^{2}+|G_{2}|^{2}\Big]\, dxdt+ \iint_{Q}|\eta_{t}|^{2} \Big[|\varphi|^{2}+|\psi|^{2} \Big]\, dxdt\right).
\end{array}
\end{equation}
Proceeding similarly, we have
\begin{equation}\label{eta-3}
\begin{array}{l}
\displaystyle \iint_{Q}\left(|\eta \varphi_{t}|^{2}+|\eta\psi_{t}|^{2}+|\eta\Delta \varphi|^{2}+|\eta\Delta\psi|^{2}\right)\, dxdt\\
 	\noalign{\smallskip}\phantom{QQgggg}
\displaystyle \leq C\left( \iint_{Q}|\eta|^{2}\Big[|G_{1}|^{2}+|G_{2}|^{2}\Big]\,dxdt \right.\\
 	\noalign{\smallskip}\phantom{QQWWWQ}
\displaystyle \left.+\iint_{Q}|\eta_{t}|^{2} \Big[|\Delta\varphi|^{2}+|\Delta\psi|^{2}+|\varphi_{t}|^{2}+|\psi_{t}|^{2}\Big]\, dxdt
\right).
\end{array}
\end{equation}
The equations \eqref{eta-1}-\eqref{eta-3} and the properties of $\eta$, we deduce that
$$
\begin{array}{l}
\overline{I}_{2}(s,\lambda;\varphi)+\overline{I}_{2}(s,\lambda;\psi)\\
 	\noalign{\smallskip}\phantom{QQWWQ}
\displaystyle \leq C \left(\iint_{\Omega \times \left(\frac{T}{4},\frac{T}{2} \right)} \Big[|\Delta \varphi|^{2}+|\Delta \psi|^{2}+|\varphi_{t}|^{2}+|\psi_{t}|^{2}\Big]\, dxdt\right. \\
 	\noalign{\smallskip}\phantom{QQWWEEEW}
\displaystyle \left. + \iint_{\Omega \times \left(\frac{T}{2},T \right)}\Big[|G_{1}|^{2}+|G_{2}|^{2}\Big]\, dxdt \right).
\end{array}
$$
Consequently, 
$$\displaystyle \overline{I}_{2}(s,\lambda;\varphi)+\overline{I}_{2}(s,\lambda;\psi) \leq C(\overline{RS}).$$
 This last together  with \eqref{eta-0} allow us obtain the inequality \eqref{eq-carl-2}.

 This ends the proof.
\end{proof}

\begin{remark}\label{d-remark-1}
If $\lambda > \dfrac{1}{\|\eta^{0}\|_{\infty}}$ is sufficiently large and denoting by   
$$\displaystyle \beta(x,t)=\frac{2}{5}\overline{\sigma}(x,t),\,\,  \displaystyle \beta^{\ast}(t)=\max_{x\in \overline{\Omega}}\beta(x,t)\, \ \mbox{ and }\,\, \hat{\beta}(t)=\min_{x\in \overline{\Omega}}\beta(x,t),$$ 
we have
$$
\hat{\beta}(t)\leq \beta(x,t)\leq \frac{5}{4}\hat{\beta}(t)\,\,\, \mbox{ and }\,\,\, \frac{4}{5}\beta^{\ast}(t)\leq \beta(x,t)\leq \beta^{\ast}(t).
$$
 Then, due to Proposition \ref{carl-2} we have

\begin{equation}\label{id-remark-1}
\begin{array}{l}
\displaystyle I_{\ast}(s,\lambda;\varphi)+I_{\ast}(s,\lambda;\psi)\leq C\left(\iint_{Q}e^{-4s\beta^{\ast}}\left[\lambda^{4}(s\hat{\xi})^{3}|G_{1}|^{2}+|G_{2}|^{2}\right]\, dxdt\right.\\
\noalign{\smallskip} 
	\phantom{ I(s,\lambda;\varphi)+I(s,\lambda;\psi)233Qrt}
\left. \displaystyle +\iint_{\omega\times (0,T)}e^{-4s\beta^{\ast}}\lambda^{8}(s\hat{\xi})^{7}|\varphi|^{2}\, dxdt 
\right),
\end{array}
\end{equation}
where 
$$
I_{\ast}(s,\lambda; \varphi)=\iint_{Q}e^{-5s\beta^{\ast}}\left[(s\hat{\xi})^{-1}(|\phi_{t}|^{2}+|\Delta \phi|^{2})+\lambda^{2}s\xi^{\ast}|\nabla \phi|^{2}+\lambda^{4}(s\xi^{\ast})^{3}|\phi|^{2}\right]\, dxdt.
$$
\end{remark}

\begin{remark}\label{remark-1}
The new weight functions $\tilde{\sigma}(x,t)$ and $\tilde{\xi}(x,t)$ will not be really necessary, because due to the exponential blow-up of the weight functions $\sigma(x,t)$ and $\xi(x,t)$ we will prove that $h(x,0)=y(x,T)=0$, furthermore, $h(x,T)=y(x,0)=0$, but this are the initial condition (in this case there is not a contradiction respect to the initial values), then it is possible to assume the usual weight functions $\sigma(x,t)$ and $\xi(x,t)$.
\end{remark}

\begin{remark}\label{remark-1}
It will not be really necessary employ the weight functions $\sigma^*(t)$, $\hat{\sigma}(t)$, $\xi^*(t)$ and $\hat{\xi}(t)$, we can apply the usual weight functions $\sigma(x,t)$ and $\xi(x,t)$ to conclude the desire result, but in this case, the estimates will be  tedious and bigger, so, for simplicity we will take the maximum and minimum of this weight functions with respect to space variable.

\end{remark}

\subsection{Null controllability for the linear system \eqref{linear-1}}

Assume that the  notations and hypothesis of Remark \ref{d-remark-1} hold. Then, in order to simplify the notation, we will denote by
\begin{equation}
\label{weight-1}
\begin{array}{c}
\rho_0 := e^{2s\beta^{\ast}}{\xi^{\ast}}^{-\frac{3}{2}}, \ \ \rho_1 := e^{2s\beta^{\ast}} {\xi^{\ast}}^{-\frac{7}{2}}, \ \ \rho:=e^{\frac{5}{2}s\beta^{\ast}},\\
\hat{\rho}_{0}:=e^{\frac{3}{2}s\beta^{\ast}}{\xi^{\ast}}^{-\frac{9}{2}}, \  \  \hat{\rho}_{k}:=e^{\frac{3}{2}s\beta^{\ast}}{\xi^{\ast}}^{-\frac{15+2k}{2}}, \, k \in \mathbb{N}.
\end{array}
\end{equation}
Thanks to Proposition \ref{carl-2}, we will be able to prove the null controllability of \eqref{linear-1} for right-hand $f$ and $g$ that  decay sufficiently fast to zero as $t\to 0^{+}.$ Indeed, one has the following:
\begin{proposition}\label{p-null-c-1}
Assume the hypothesis of Proposition \ref{carl-2} and let $g_1, g_2$ satisfy $\rho g_1\, \rho_{0} g_2\in L^{2}(Q).$ Then,  the  system \eqref{linear-1}  is  null-controllable at  time $t=0$. More  precisely, there exists $u\in L^{2}(\omega\times (0,T))$ with  $\mathrm{supp}\, u\subset \omega\times[0,T]$  such that, if $(y,h)$ is the solution  of \eqref{linear-1}, one has:
\begin{description}
\item[a)]

\begin{equation}\label{eq-null-1}
	\begin{array}{c} 
		\displaystyle \iint_{\omega\times(0,T)}\Big[\hat{\rho}^{2}_{1}|u_{t}|^{2}+\hat{\rho}^{2}_{1}|\Delta u|^{2}+\hat{\rho}^{2}_{1}|u|^{2}\Big]\,dxdt\leq C\iint_{Q}\Big[\rho^{2}|g_1|^{2}+\rho_0^2 |g_2|^2\Big]\,dxdt,\\
		\noalign{\smallskip}
		\displaystyle \iint_{Q}\left(\rho^{2}_{0}|y|^{2}+\rho^{2}_{0}|h|^{2}\right)\,dxdt<+\infty,\,\\
	\end{array}
\end{equation}

	\item[b)] 
\begin{equation}
\label{eq-null-2}
\begin{array}{l}
\displaystyle \underset{t\in [0,T]}{\sup} \left(\hat{\rho}_{0}^2(t)\int_{\Omega} |y(t)|^{2}\, dx \right) +\underset{t\in [0,T]}{\sup} \left(\hat{\rho}_{0}^2(t)\int_{\Omega}| h(t)|^{2}\, dx \right)\\
\noalign{\smallskip} 
	\displaystyle + \iint_{Q}\hat{\rho}^{2}_{0}(|\nabla y|^{2}+|\nabla h|^{2})\ dxdt \leq C \left(\iint_Q \left(\rho^2 |y|^2+ \rho_0^2 |h|^2\right)\,dx\,dt\right. \\
\noalign{\smallskip} 
	\phantom{(\nabla)}
  \displaystyle +\iint_{Q}\left(\rho^{2}|g_1|^{2}+\rho_0^2 |g_2|^2\right)\,dx\,dt + \left. \iint_{\omega\times(0,T)} \rho_1^2 |u|^2 \,dx\,dt\right),
\end{array}
\end{equation}

\item[c)]
\begin{equation}\label{est-2}
\begin{array}{l}
\displaystyle \underset{t\in [0,T]}{\sup} \left(\hat{\rho}_{1}^2(t) \int_{\Omega}|\nabla y(t)|^2\, dx\right) +\underset{t\in [0,T]}{\sup} \left(\hat{\rho}_{1}^2(t) \int_{\Omega}|\nabla h(t)|^2\, dx\right)\\
\noalign{\smallskip} 
\displaystyle
+\iint_{Q}\hat{\rho}_{1}^2(| \Delta y|^2+|\Delta h|^{2})\, dxdt+\int_{Q}\hat{\rho}^{2}_{1}(|y_{t}|^{2}+|h_{t}|^{2})\,dx\,dt\\
\noalign{\smallskip} 
	\phantom{E}
\displaystyle \leq C \left(\iint_Q\left( \rho^2 |y|^2+ \rho_0^2 |h|^2\right)\,dx\,dt+ \iint_{\omega\times(0,T)} \rho_1^2 |u|^2 \,dx\,dt \right. \\
\noalign{\smallskip} 
	\phantom{(\nabla)QWERR}
\displaystyle
\left.  +\iint_{Q}\left(\rho^{2}|g_1|^{2}+\rho_0^2 |g_2|^2\right)\,dx\,dt \right).
\end{array}
\end{equation}
\end{description}

\end{proposition}
\begin{proof}
Let us introduce the following constrained extremal problem:
\begin{equation}\label{mini-1}
\left\{\begin{array}{l}
\displaystyle\inf \left\{ \frac{1}{2}\left(\iint_{Q}\rho^{2}_{0}(|y|^{2}+|h|^{2})\, dxdt+\iint_{\omega\times (0,T)}\rho^{2}_{1}|u|^{2}\, dxdt \right) \right\},\\
\mbox{subject to\,  } u\in L^{2}(Q),\  \mathrm{supp}\ u\subset \omega\times [0,T]\,\, \mbox{ and}\\
\left\{\begin{array}{lll}
y_{t}-a(0) \Delta y+A y=u\chi_{\omega}+g_1 &\mbox{in} & Q,\\
	\noalign{\smallskip} 
-h_{t}-a(0)\Delta h+A h=y\chi_{\mathcal{O}}+g_2 & \mbox{in} & Q,\\ 
	\noalign{\smallskip} 
y=0,\, h=0 & \mbox{on} & \Sigma,\\
	\noalign{\smallskip} 
y(x,0)=0,\, h(x,T)=0 &\mbox{in} & \Omega.
\end{array}\right.
\end{array}\right.
\end{equation}
Assume  that this problem admits a unique solution $(\hat{y},\hat{h}).$  Then, in virtue of the Lagrange's principle there exist dual variables $(\overline{y},\overline{h})$ such that:
\begin{equation}\label{mini-2}
\left\{\begin{array}{lll}
\hat{y}=\rho^{-2}_{0}\left(\bar{y}_{t}-a(0) \Delta \bar{y}+A \bar{y}-\bar{h}\chi_{\mathcal{O}} \right) &\mbox{in} & Q,\\
	\noalign{\smallskip} 
\hat{h}=\rho^{-2}_{0}\left(-\bar{h}_{t}-a(0)\Delta \bar{h}+A \bar{h}\right) & \mbox{in} & Q,\\
	\noalign{\smallskip} 
\hat{u}=-\rho^{-2}_{1} \bar{y} & \mbox{in} & \omega\times (0,T),\\
	\noalign{\smallskip} 
\hat{y}=\hat{h}=0 & \mbox{on} & \Sigma. 
\end{array}\right.
\end{equation}
Let $\chi\in C^{\infty}_{0}(\omega),\, 0 \leq \chi \leq 1,\,\, \mbox{with}\,\, \chi|_{\omega_{0}}=1$,   we set
$$
\mathcal{P}_{0} = \Big\{(y,h)\in [C^{\infty}(\overline{Q})]^{2};\, y=0,\, h=0 \,\, \mbox{on}\,\, \Sigma,\, y(0)=0,\, h(T)=0 \,\, \mbox{in}\,\, \Omega \Big\},
$$
and
\begin{equation}\label{lax-1}
\begin{array}{l}
\displaystyle a\left((\bar{y},\bar{h});(y,h)\right)=\iint_{Q}\rho^{-2}_{0}(L^{\ast}\bar{y}-\bar{h}\chi_{\mathcal{O}})(L^{\ast}y-h\chi_{\mathcal{O}})\, dxdt\\
	\noalign{\smallskip} 
	\phantom{a(TTT)}
\displaystyle + \iint_{Q}\rho^{-2}_{0}\tilde{L} \bar{h} \tilde{L} h\, dxdt+\iint_{\omega\times (0,T)}\chi^{2}\rho^{-2}_{1}\bar{y}y\, dxdt,\\
	\noalign{\smallskip} 
	\phantom{a(TTT)DDrrDD} \forall\, (y,h)\in \mathcal{P}_{0},
\end{array}
\end{equation}
where $L^* y := -y_t - a(0) \Delta y + A y$\ \ and\ \ $\tilde{L} h := h_t - a(0) \Delta h + A h$.\\
By definition \eqref{lax-1}, one can see that, if the functions $\hat{y},\, \hat{h}$ solves \eqref{mini-1}, we must have
\begin{equation}\label{mini-3}
a\left((\bar{y},\bar{h});(y,h)\right)=\mathcal{G}(y,h), \ \, \forall \,  (y,h)\in \mathcal{P}_{0} ,
\end{equation}
where
\begin{equation}\label{mini-4}
\mathcal{G}(y,h)=\iint_{Q}g_{1}y\, dxdt+\iint_{Q}g_{2}h\, dxdt.
\end{equation}
The main idea is to prove that there exists exactly one $(\bar{y},\bar{h})$ satisfying \eqref{mini-3}. Then we will define $(\hat{y},\hat{h})$ using \eqref{lax-1} and we will check that it fulfills the desired properties. Indeed, observe that  the  inequality \eqref{id-remark-1} holds for $(y,h)\in \mathcal{P}_{0},$
\begin{equation}\label{mini-5}
\iint_{Q}\rho^{-2}_{0}|y|^{2}\, dxdt+\iint_{Q}\rho^{-2}_{0}|h|^{2}\, dxdt\leq C a\left((y,h);(y,h)\right), \,\,\, \forall \, (y,h)\in \mathcal{P}_{0}.
\end{equation}
In the linear space $\mathcal{P}_{0}$ we consider the bilinear form $a(\cdot,\cdot)$ given by \eqref{lax-1}; from the unique continuation given in \eqref{eq-carl-2} we deduce that $a(\cdot,\cdot)$ is a scalar product in $\mathcal{P}_{0}.$ Let us now consider  the space $\mathcal{P},$ given by the completion of $\mathcal{P}_{0}$ for the norm  associated to $a(\cdot,\cdot).$ This is a Hilbert space and $a(\cdot,\cdot)$ is a continuous and coercive  bilinear form on $\mathcal{P}.$

We turn to the linear operator $\mathcal{G},$ given by \eqref{mini-4} for all $(y,h)\in \mathcal{P},$ a simple computation leads to 
$$
\mathcal{G}(y,h) \leq \|\rho_{0}g_{1}\|_{L^{2}(Q)} \|\rho^{-1}_{0}y\|_{L^{2}(Q)} + \|\rho_{0}g_{2}\|_{L^{2}(Q)} \|\rho^{-1}h\|_{L^{2}(Q)}. 
$$
Then, using \eqref{mini-5} and the density of $\mathcal{P}$ in $\mathcal{P}_{0},$ we have 
$$
\mathcal{G}(y,h)\leq C \left(\|\rho_{0}g_1\|_{L^{2}(Q)} + \|\rho_{0}g_2\|_{L^{2}(Q)} \right)\|(y,h)\|_{\mathcal{P}},\,\,\forall \, (y,h)\in \mathcal{P}.
$$
Consequently $\mathcal{G}$ is a bounded linear operator on $\mathcal{P}.$ Then, in view of Lax-Milgram's Lemma, there exists one and only one $(\bar{y},\bar{h})$ satisfying
\begin{equation}\label{mini-6}
\left\{\begin{array}{l}
a\left((\bar{y},\bar{h});(y,h)\right)=\mathcal{G}(y,h),\,\, \forall \, (y,h)\in \mathcal{P},\\
	\noalign{\smallskip} 
(\bar{y},\bar{h})\in \mathcal{P}.
\end{array}\right.
\end{equation}
We finally get the existence of $(\hat{y},\hat{h}),$ just setting
$$
\hat{y}=\rho^{-2}_{0}\left(L^{\ast}\bar{y}-\bar{h}\chi_{\mathcal{O}}\right),\, \hat{h}=\rho^{-2}_{0}\tilde{L}\bar{h}\,\,\mbox{and}\,\, \hat{u}=\chi\rho^{-2}_{1}\bar{y}.
$$
We see that $(\hat{y},\hat{h})$ solves \eqref{linear-1} and since  that $(\hat{y},\hat{h})\in \mathcal{P}$ we have
\begin{equation}\label{mini-7}
	\begin{array}{l}
		\displaystyle \iint_{Q}\rho^{2}_{0}|\hat{y}|^{2}\, dxdt+ \iint_{Q}\rho^{2}_{0}|\hat{h}|^{2}\, dxdt+ \iint_{\omega\times (0,T)}\rho^{2}_{1}|\hat{u}|^{2}\, dxdt\\
		\noalign{\smallskip} 
		\phantom{a(TTT)DDD}
		\displaystyle =a\left((\bar{y},\bar{h});(\bar{y},\bar{h})\right) \leq C\iint_{Q}\Big[\rho^{2}|g_1|^{2}+\rho_0^2 |g_2|^2\Big]\,dxdt<\infty.
	\end{array}
\end{equation}
Now, let us prove the items $a),\, b)$ and $c)$.

\underline{\it Proof of $a)$:} 

First, let us denote the functions $y^{\ast},\, h^{\ast},\, F^{\ast}_{1}$ and $F^{\ast}_{2}$ the following way\,\,  $\displaystyle y^{\ast}=\hat{\rho}_{1}\rho^{-2}_{1}\bar{y},$  \,$\displaystyle h^{\ast}=\hat{\rho}_{1}\rho^{-2}_{1}\bar{h},\,\, F^{\ast}_{1}=-\hat{\rho}\rho^{-2}_{1}(L^{\ast}\bar{y}-\bar{h}\chi_{\cal O})-(\hat{\rho}_{1}\rho^{-2}_{1})_{t}\bar{y}\,\,\, \mbox{and}\,\,\, F^{\ast}_{2}=\hat{\rho}_{1}\rho^{-2}_{1}\tilde{L}\bar{h}+(\hat{\rho}_{1}\rho^{-2}_{1})_{t}\bar{h},$ \,
and by the equation \eqref{adjoint-1}, we have that $(y^{\ast}, h^{\ast})$ solves the following system
\begin{equation}\label{mini-8}
\left\{\begin{array}{lll}
-y^{\ast}_{t}-a(0)\Delta y^{\ast}+Ay^{\ast}=h^{\ast}\chi_{\mathcal{O}}+F^{\ast}_{1}  & \mbox{in} & Q,\\
	\noalign{\smallskip} 
h^{\ast}_{t}-a(0)\Delta h^{\ast}+A h^{\ast}=F^{\ast}_{2} & \mbox{in} & Q,\\ 
	\noalign{\smallskip} 
y^{\ast}=0,\, h^{\ast}=0 & \mbox{on} & \Sigma,\\
	\noalign{\smallskip} 
y^{\ast}(x,T)=0,\, h^{\ast}(x,0)=0 & \mbox{in} & \Omega,
\end{array}\right.
\end{equation}
thanks to \eqref{mini-7} we have 
$$
\|F^{\ast}_{1}\|^{2}_{L^{2}(Q)} + \|F^{\ast}_{2}\|^{2}_{L^{2}(Q)}\leq a\left((\bar{y},\bar{h});(\bar{y},\bar{h})\right).
$$
Then,  
$$\displaystyle y^{\ast}\in L^{2}(0,T;H^{2}(\Omega))\,\,\mbox{ and}\,\,  \displaystyle y^{\ast}_{t}\in L^{2}(0,T;L^{2}(\Omega)).$$
 From  definition of  $\hat{u}$, we have    
$$\hat{\rho}_{1}\hat{u}\in L^{2}(0,T;H^{2}(\Omega)),\,  \hat{\rho}_{1}\hat{u}_{t}\in L^{2}(Q),$$
with 
$$
\iint_{\omega\times (0,T)}\hat{\rho}^{2}_{1}\left(|\hat{u}_{t}|^{2}+|\Delta \hat{u}|^{2}\right)\, dxdt\leq C \iint_{\omega\times (0,T)}\rho^{2}_{1}|\hat{u}|^{2}\, dxdt.
$$
\underline{\it Proof of $b)$:}

To this purpose, let multiplying the first PDE in \eqref{linear-1} by  by $\hat{\rho}^{2}y$ and  let us integrate in $\Omega.$ We obtain:
$$
\int_{\Omega}\hat{\rho}^{2}_{0}y\left(y_{t}-a(0)\Delta y+Ay\right)\, dx =\int_{\Omega}\hat{\rho}^{2}_{0}y(u\chi_{\omega}+g_1)\, dx.
$$
Notice that
$$
\begin{array}{l}
\displaystyle \frac{1}{2}\int_{\Omega}\hat{\rho}^{2}_{0}|y|^{2}\, dx+\frac{a(0)}{2}\int_{\Omega}\hat{\rho}^{2}_{0}|\nabla y|^{2}\, dx \displaystyle \leq C\left( \int_{\Omega}\hat{\rho}^{2}_{0}|y|^{2}+ \hat{\rho}^{2}_{0}|g_1|^{2}\, dx+\int_{\omega}\hat{\rho}^{2}_{1}|u|^{2}\, dx \right),
\end{array}
$$
consequently,
\begin{equation}\label{mini-9}
\begin{array}{l}
\displaystyle \sup_{t\in[0,T]}\left(\hat{\rho}^{2}_{0}(t)\int_{\Omega}|y(t)|^{2}\, dx\right)+\iint_{Q}\hat{\rho}^{2}_{0}|\nabla y|^{2}\, dxdt\\
\noalign{\smallskip}\phantom{QQQQQDDDDqq}
\displaystyle \leq C\left( \iint_{Q}\hat{\rho}^{2}_{0}|y|^{2}+ \hat{\rho}^{2}_{0}|g_1|^{2}\, dxdt+\iint_{\omega\times (0,T)}\hat{\rho}^{2}_{1}|u|^{2}\, dxdt \right).
\end{array}
\end{equation}
Similarly, after of multiplying the second equation in \eqref{linear-1}  by $\hat{\rho}^{2} h$  and integrate in Q. We have
\begin{equation}\label{mini-10}
\begin{array}{l}
\displaystyle \sup_{t\in[0,T]}\left(\hat{\rho}^{2}_{0}(t)\int_{\Omega}|h(t)|^{2}\, dx\right)+\iint_{Q}\hat{\rho}^{2}_{0}|\nabla h|^{2}\, dxdt\\
\noalign{\smallskip}\phantom{QQQQQDDDDqq}
\displaystyle \leq C\left( \iint_{Q}\hat{\rho}^{2}_{0}|y|^{2}+\rho^{2}_{0}|h|^{2} \,dxdt+ \iint_{Q}\hat{\rho}^{2}_{0}|g_2|^{2}\, dxdt \right).
\end{array}
\end{equation}
Then, from \eqref{mini-9}-\eqref{mini-10} we have that \eqref{eq-null-2} holds.

\underline{\it Proof of $c):$}

Proceeding as before, let us multiplying the first and second equation in \eqref{linear-1} by terms $\hat{\rho}^{2}_{1}y_{t}$ and $\hat{\rho}^{2}_{1}h_{t}$, integrating in $\Omega$  and using the properties of weights $\hat{\rho}_{1}, \, \hat{\rho}_{1,t}$ as above. We have
\begin{equation}\label{mini-11}
\begin{array}{l}
\displaystyle \sup_{t\in [0,T]}\left(\hat{\rho}^{2}_{1}(t)\int_{\Omega}|\nabla y(t)|^{2}\, dx\right)+\sup_{t\in [0,T]}\left(\hat{\rho}^{2}_{1}(t)\int_{\Omega}|\nabla h(t)|^{2}\, dx\right)\\
\noalign{\smallskip}\phantom{QQQQQDDDDqq}
\displaystyle +\iint_{Q}\hat{\rho}^{2}_{1}\left(|y_{t}|^{2}+|h_{t}|^{2} \right)\, dxdt \leq C(RS),
\end{array}
\end{equation}
where $(RS)$ represent the term in the right side of \eqref{est-2}.

Continuing, let us multiplying the equations in \eqref{linear-1} by $\hat{\rho}^{2}_{1}\Delta y$ and $\hat{\rho}^{2}_{1}\Delta h$, after let us integrate in $\Omega$ and due to estimate \eqref{mini-11} we get the following inequality 
\begin{equation}\label{mini-12}
\iint_{Q}\hat{\rho}^{2}_{1}\left(|\Delta y|^{2}+| \Delta h|^{2} \right) \, dxdt \leq C(RS).
\end{equation}
From the estimates \eqref{mini-9}-\eqref{mini-12}, the inequality \eqref{est-2} holds.

\end{proof}

\subsection{Some additional estimates of the state}

The next results provide additional properties of the states found in Proposition \ref{p-null-c-1}. They will be needed below, in Section 4.
\begin{proposition}\label{p-null-c-2}
Let the hypothesis in Proposition \ref{p-null-c-1} be satisfied and  let the state-control $(y,h,u)$ of \eqref{linear-1} satisfying \eqref{eq-null-1}. Then
\begin{description}
 
\item[a)] If\,  $\hat{\rho}_{1}g_{1,t} \in L^{2}(Q),$ we have

\begin{equation}
\label{eq-null-3}
\begin{array}{l}
 \displaystyle \iint_{Q}\hat{\rho}^{2}_{3} (|y_{tt}|^{2}+|\Delta y_{t}|^{2}+|\nabla y_{t}|^{2}) \,dx\,dt +
   \underset{t\in [0,T]}{\sup} \left(\hat{\rho}^{2}_{2}(t) \int_{\Omega}|\Delta y(t)|^2\, dx\right)\\
\noalign{\smallskip} 
	\phantom{()}
\displaystyle 
+\underset{t\in [0,T]}{\sup} \left(\hat{\rho}^{2}_{3}(t) \int_{\Omega}|\nabla y_{t}(t)|^2\, dx\right)
 \leq C \left( \iint_Q (\rho^2 |g_1|^2 +\rho_0^2 |g_2|^2+\hat{\rho}^{2}_{1}|g_{1,t}|^{2})\,dx\,dt \right.\\
\noalign{\smallskip} 
	\phantom{(Q)}
 \displaystyle \left. 
+\iint_{\omega\times(0,T)} \rho^{2}_{1}|u|^{2}+\rho_1^2 |u_{t}|^2 \,dx\,dt + \iint_Q \rho_0^2 |y|^2+\rho^{2}_{0}|h|^{2} \,dx\,dt \right), \\
\end{array}
\end{equation}

\item[b)] If \,   $\hat{\rho}_{1}g_{1,t} \in L^{2}(Q)$ and  $ \hat{\rho}_{4}g_{1} \in L^{2}(0,T;H^{2}(\Omega)\cap H^{1}_{0}(\Omega)),$ we have

\begin{equation}
\label{eq-null-4}
\begin{array}{l}
\displaystyle \iint_{Q}\hat{\rho}^{2}_{4}(|\Delta^{2}y|^{2}+|\nabla \Delta y|^{2})\, dxdt+ \underset{t\in [0,T]}{\sup} \left(\hat{\rho}^{2}_{5}(t) \int_{\Omega}|\nabla \Delta y(t)|^2\, dx\right)\\
\noalign{\smallskip} 
	\phantom{( )}
\leq \displaystyle  C \left( \iint_Q (\rho^2 |g_1|^2 +\rho_0^2 |g_2|^2)\,dx\,dt +\iint_{\omega\times(0,T)} \rho^{2}_{1}|u|^{2}+\rho_1^2 |u_{t}|^2 \,dx\,dt \right.\\
\noalign{\smallskip} 
	\phantom{N^2(\tilde{W} )W}
 \displaystyle \left. + \iint_{Q}\hat{\rho}^{2}_{1}|g_{1,t}|^{2}+\hat{\rho}^{2}_{4}|\Delta g_{1}|^{2}\,dx\,dt + \iint_Q \rho_0^2 |y|^2+\rho^{2}_{0}|h|^{2} \,dx\,dt \right). \\
\end{array}
\end{equation}
\end{description}
\end{proposition}
\begin{proof}
Due to Proposition \ref{p-null-c-1}, we have that there exists an state-control $(y,h,u)$ of \eqref{linear-1} satisfying \eqref{eq-null-1}, \eqref{eq-null-2} and \eqref{est-2}. So, now on we are going to prove that this solution $y,h,u$  also satisfy the estimations $a)$ and  $b).$

\underline{\it Proof of a):}
In order to prove the estimate \eqref{eq-null-3}, let us derivative the first equation in \eqref{linear-1} with respect the time variable.
\begin{equation}\label{n-mini-1}
y_{tt}-a(0)  \Delta y_{t}+Ay_{t}=u_{t}\chi_{\omega}+g_{1,t}\,\, \mbox{in}\,\, Q.
\end{equation}
In the following, let us prove similar estimates as in Proposition \ref{p-null-c-1}.
\begin{itemize}
	\item 
Multiplying in \eqref{n-mini-1} by $\hat{\rho}^{2}y_{t}$ and integrating in $Q$ and  from 
 Proposition \ref{p-null-c-1} we have
\begin{equation}\label{n-mini-2}
\sup_{t\in[0,T]}\left(\hat{\rho}^{2}_{2}(t)\int_{\Omega}|y_{t}(t)|^{2}\, dx \right)+\iint_{Q}\hat{\rho}^{2}_{2}|\nabla y_{t}|^{2}\, dxdt\leq C(RS)_{1},
\end{equation}
here $(RS)_{1}$ is the right side  in \eqref{eq-null-3}.

\item
Let us multiplying by $\hat{\rho}^{2}_{2}\Delta y$ in the first equation in \eqref{linear-1} and integrating in $Q$. Thanks the estimates \eqref{n-mini-2}, we have
\begin{equation}\label{n-mini-3}
\sup_{t\in[0,T]}\left(\hat{\rho}^{2}_{2}(t)\int_{\Omega}|\Delta y(t)|^{2}\, dx\right)\leq C(RS)_{1}.
\end{equation}

\item
Multiplying by $\hat{\rho}^{2}_{3}y_{tt}$ in \eqref{n-mini-1} and integrating in $Q$.
\begin{equation}\label{n-mini-4}
\sup_{t\in[0,T]}\left(\hat{\rho}^{2}_{3}(t)\int_{\Omega}|\nabla y_{t}(t)|^{2}\, dx\right)+\iint_{Q}\hat{\rho}^{2}_{3}|y_{tt}|^{2}\, dxdt\leq C(RS)_{1}.
\end{equation}

\item
Multiplying by $-\hat{\rho}^{2}_{3}\Delta y_{t}$ in \eqref{n-mini-1} and integrating in $Q$.
\begin{equation}\label{n-mini-5}
\iint_{Q}\hat{\rho}^{2}_{3}|\Delta y_{t}|^{2}\, dxdt\leq C(RS)_{1}.
\end{equation}
\end{itemize}
Therefore from the previous estimates \eqref{n-mini-2}-\eqref{n-mini-5}, we conclude that \eqref{eq-null-2} holds.

\underline{\it Proof of $b)$:}

To prove this item, let us apply the Laplacian operator $\Delta $ in the system \eqref{linear-1} and get the following new system:
\begin{equation}
\left\{\begin{array}{lll}
\hat{y}_{t}-a(0)\Delta \hat{y}+A\hat{y}=\Delta (u\chi_{\omega})+\Delta g_{1} & \mbox{in} & Q,\\
\noalign{\smallskip}
\hat{y}=0 & \mbox{on} & \Sigma,\\
\noalign{\smallskip}
\hat{y}(x,0)=0 & \mbox{in} &\Omega.
\end{array}\right.
\end{equation}
Where $\hat{y}=\Delta y$, in the way similar as in part of estimates $a)$

\begin{itemize}
  \item
	
	Multiplying by $\hat{\rho}^{2}_{4} \hat{y}$, integrating in $Q$ and using the previous  estimates of $a)$. We have
	\begin{equation}\label{b-mini-1}
	\iint_{Q}\hat{\rho}^{2}_{4}|\nabla \hat{y}|^{2}\, dxdt\leq C(RS)_{2},
	\end{equation}
	where  $(RS)_2$ is the right side of \eqref{eq-null-4}.

	\item 
	
	Multiplying by $-\hat{\rho}^{2}_{4} \Delta \hat{y}$, integrating in $Q$ and using the previous  estimates of $a)$. We have
	\begin{equation}\label{b-mini-2}
	\iint_{Q}\hat{\rho}^{2}_{4}|\Delta \hat{y}|^{2}\, dxdt\leq C(RS)_{2}.
	\end{equation}
   	
	\item  
	
	Multiplying by  $\hat{\rho}^{2}_{5} \hat{y}_{t}$, integrating in $Q$ and using the previous  estimates of $a)$ join with \eqref{b-mini-1}-\eqref{b-mini-2}. We have
	\begin{equation}\label{b-mini-3}
	\sup_{t\in[0,T]}\left(\hat{\rho}^{2}_{5}(t)\int_{\Omega}|\nabla \hat{y}(t)|^{2}\, dx\right)\leq C(RS)_{2}.
	\end{equation}
\end{itemize}
	Therefore, the estimates \eqref{b-mini-1}-\eqref{b-mini-3} the item $b)$ holds.

Thus, we conclude the proof.
\end{proof}


\section{Proof of the  main result}\label{sec4}
\setcounter{equation}{0}

\subsection{Locally null controllability  of  optimal system \eqref{eq-char}}

In this subsection we prove the null controllability for the optimal system using Liusternik's theorem.
\begin{theorem}[\textbf{Liusternik's Theorem}]
\label{Liu-1}
Let $E$ and $F$ be Banach spaces and let $\mathcal{A}:  B_r(0) \subset E \to F$ be a $C^1$ mapping. Let us assume that the derivative $\mathcal{A}'(0): E \to F$ is onto and let us denote set $\xi_0=\mathcal{A}(0)$. Then, there exist $\epsilon>0$, a mapping $W: B_\epsilon(\xi_0)\subset F \to E$ and a constant $K>0$ satisfying
$$W(z) \in B_r(0) \ \text{and} \ \ \mathcal{A}(W(z))=z, \, \ \ \forall z \in B_\epsilon(\xi_0),$$
$$\|W(z)\|_E\leq K\|z-\xi_0\|_F\, \ \ \forall z \in B_\epsilon(\xi_0).$$
\end{theorem}
The proof of this Theorem can be found in \cite{Ale1}.

Now, let us introduce the space
$$
\begin{array}{l}
E=\Big\{ (y,h,u):\, \rho_{0} y,\, \rho_{0} h,\, y_{x_{i}},\, y_{x_{i}x_{j}}\, \in L^{2}(Q);\, \rho u,\, \hat{\rho}_{1}u_{t},\, \hat{\rho}_{3}\Delta u\in L^{2}(\omega\times (0,T)),\Big.\\
\noalign{\smallskip}\phantom{QQQ}
\displaystyle \, \rho (\mathcal{L}_{1}y-u\chi_{\omega}),\, \hat{\rho}_{1}(\mathcal{L}_{1}y_{t}-u_{t}\chi_{\omega}),\, \hat{\rho}_{3}(\mathcal{L}_{1}\Delta y-\Delta(u\chi_{\omega})),\, \rho (\mathcal{L}_{2}h-y\chi_{\mathcal{O}})\in L^{2}(Q),\\
\noalign{\smallskip}\phantom{QQQ}
\displaystyle
\, \Big. \, y|_{\Sigma}=0,\, h|_{\Sigma}=0,\, h(T)= 0\Big\},
\end{array}
$$
where $\mathcal{L}_{1}$ and $\mathcal{L}_{2}$  denote the following expressions:
$$
\begin{array}{l}
\mathcal{L}_{1}y=y_{t}-a(0)\Delta y+Ay,\,\, \mathcal{L}_{2}h=-h_{t}-a(0)\Delta h+Ah,
\end{array}
$$
and  the norm in $E$ is 
$$
\begin{array}{l}
	\|(y,h,u)\|^{2}_{E}=\|\rho_{0} y\|^{2}_{L^{2}(Q)} + \|\rho_{0} h\|^{2}_{L^{2}(Q)} + \|\hat{\rho} u\|^{2}_{L^{2}(\omega\times (0,T))}\\
\noalign{\smallskip}\phantom{QQQQTTQ}
\displaystyle +\|\rho (\mathcal{L}_{1}y-u\chi_{\omega})\|^{2}_{L^{2}(Q)} + \|\hat{\rho}(\mathcal{L}_{1}y_{t}-u_{t}\chi_{\omega})\|^{2}_{L^{2}(Q)}\\
\noalign{\smallskip}\phantom{QQQQQTT}
\displaystyle + \|\hat{\rho}_{3}(\mathcal{L}_{1}\Delta y-\Delta(u\chi_{\omega}))\|^{2}_{L^{2}(\omega\times (0,T))} + \|\rho (\mathcal{L}_{2}h-y\chi_{\mathcal{O}} )\|^{2}_{L^{2}(Q)}.
\end{array}
$$
It is clear that $E$ is a Banach space for the norm $\|\cdot\|_E$.

\begin{remark}\label{est-prop}
Notice that, if $(y,h,u)\in E$, in view of Proposition \ref{p-null-c-1} and \ref{p-null-c-2}, one has
$$
\begin{array}{l}
	\|\hat{\rho}_{4} \Delta y\|^{2}_{L^{\infty}(0,T;H^{1}_{0}(\Omega))} + \|\hat{\rho}_{1}h \|^{2}_{L^{\infty}(0,T;H^{1}_{0}(\Omega))} + \|\hat{\rho}_{4} y\|^{2}_{X_{1}} + \|\hat{\rho}_{1}h_{t}\|^{2}_{X_{0}}\leq C \|(y,h,u)\|^{2}_{E}.\\
\end{array}
$$
\end{remark}
Also, let us define the following space
$$
\begin{array}{l}
F_{1}=\Big\{ g:\,  \rho g ,\, \hat{\rho}_{1}g_{t}\in L^{2}(Q);\, \hat{\rho}_{3}g\in L^{2}(0,T;H^{2}(\Omega)\cap H^{1}_{0}(\Omega))\Big\}, \phantom{QQQQQqqqweeww}\\
\noalign{\smallskip}
F_{2}=\Big\{ g:\,  \rho g\in L^{2}(Q)\Big\},
\end{array}
$$
and  denote by $F = F_{1}\times F_{2},$ with the  norm 
$$\|(g_1,g_2)\|^{2}_{F}:=\iint_{Q} \Big[|\rho g_1|^{2}+|\hat{\rho}_{1}g_{1,t}|^{2}+|\hat{\rho}_{3}\Delta g_1|^{2}\Big]\, dxdt+\iint_{Q}|\rho g_2|^{2}\, dxdt.$$
It is clear that  $F$  is Banach space with this norm.

Now, let us define  the  mapping  $\mathcal{A}:E \mapsto F,$ giving by $\mathcal{A}=(\mathcal{A}_{1},\mathcal{A}_{2})$ where
\begin{equation}\label{def-A}
\left\{\begin{array}{l}
\mathcal{A}_{1}(y,h,u)=y_{t}-\nabla \cdot \left(a(\nabla y)\nabla y\right)+f(y)-u\chi_{\omega},\phantom{QQQQQqq}\\
\noalign{\smallskip}
\mathcal{A}_{2}(y,h,u)=-h_{t}-\nabla\cdot\left[D a(\nabla y)(\nabla y\cdot\nabla h)+a(\nabla y)\nabla h \right]+f'(y)h-y\chi_{\mathcal{O}}.
\end{array}\right.
\end{equation}

We  will prove that  there exists $\epsilon>0$  such  that, if $(g_{1},g_{2})\in F$  and  $\|(g_{1},g_{2})\|_{F}\leq \epsilon$, then  the equation
\begin{equation}\label{H-eq1}
{\cal A}(y,h,u)=(g_{1},g_{2}),\,\,  \ \text{where }\ (y,h,u)\in E,
\end{equation}
possesses  at least  one solution.

In particular, this  shows  that \eqref{eq-char}  is  locally nul controllable at time $t=0$ and,  furthermore, the state-control  triplets  can  be  chosen  in  $E$.  We  will  apply  Theorem  $\ref{Liu-1}$.

In order to show that Theorem $\ref{Liu-1}$ can be applied in this setting, we will use several lemmas.

\begin{lemma}\label{lem1}
Let $\mathcal{A}: E \rightarrow F$ be the mapping defined by $(\ref{def-A})$. Then, $\mathcal{A}$ is well defined and continuous.
\end{lemma}

\begin{proof}
In order to prove that $\mathcal{A}$ is well defined, we will  show that
\begin{description}
	\item [i)] $\displaystyle \mathcal{A}_{1}(y,h,u)\in F_{1}, \,\, \forall\, (y,h,u)\in E$.
	
	Let $(y,h,u)\in E$, we have
	$$
	\begin{array}{l}
	\displaystyle \iint_{Q}\rho^{2}|\mathcal{A}_{1}(y,h,u)|^{2}\, dxdt\\
				\noalign{\smallskip}\phantom{QQQQA}
	\displaystyle = \iint_{Q}\rho^{2}|y_{t}-\nabla\cdot \left(a(\nabla y)\nabla y\right)+f(y)-u\chi_{\omega} |^{2}\, dxdt\\
				\noalign{\smallskip}\phantom{QQQWA}
	\displaystyle \leq C\iint_{Q}\rho^{2}|\mathcal{L}_{1}y-u\chi_{\omega}|^{2}\, dxdt+C\iint_{Q}\rho^{2}|y|^{2}\, dxdt\\
				\noalign{\smallskip}\phantom{yWWWA}
\displaystyle 	+C \iint_{Q}\rho^{2}|\nabla\cdot\left[(a(\nabla y)-a(0))\nabla y\right]|^{2}\, dxdt\\
			\noalign{\smallskip}\phantom{QQyQWA}
=I_{1}+I_{2}+I_{3}.
	\end{array}
	$$
	From Proposition \ref{p-null-c-2} and Remark \ref{est-prop}, we see that
	$$I_{1} \leq C \|(y,h,u)\|^{2}_{E}\,\,\ \  \mbox{and}\,\,\ \ I_{2} \leq C\|(y,h,u)\|^{2}_{E}.$$
	On the other hand, since $a\in C^{3}(\mathbb{R}^N)$ and is (globally) Lipschitz continuous, one has
	$$
	\begin{array}{l}
	\displaystyle I_{3}\leq C\iint_{Q}\rho^{2}|a(\nabla y)-a(0)|^{2}|\Delta y|^{2}\, dxdt+C\iint_{Q}\rho^{2}|D_i a(\nabla y)|^{2}|\nabla y|^{2}| \Delta y|^{2}\, dxdt\\
				\noalign{\smallskip}\phantom{ I_{3}}
	\displaystyle
	\leq C\iint_{Q}\rho^{2}|\nabla y|^{2}|\Delta y|^{2}\, dxdt.
	\end{array}
	$$
	From definition of $\rho$, one has $\rho\leq C\hat{\rho}_{i}\hat{\rho}_{j}$ and $\|\cdot\|_{L^{\infty}(\Omega)}\leq C\|\cdot\|_{H^{2}(\Omega)}.$
 Then,
	$$
	\begin{array}{l}
	\displaystyle I_{3}\leq C \left( \iint_{Q}\hat{\rho}^{2}_{4}|\nabla \Delta y|^{2}\, dxdt\right) \left(\sup_{t\in[0,T]}(\hat{\rho}^{2}_{3}(t)\int_{\Omega}|\Delta y(t)|^{2}\, dx)\right)\\
				\noalign{\smallskip}\phantom{I_{3}}
	 \leq C \|(y,h,u)\|^{4}_{E}.
	\end{array}
	$$
	For all these estimates, we have
	\begin{equation}\label{A-1}
	\iint_{Q}\rho^{2}|\mathcal{A}_{1}(y,h,u)|^{2}\, dxdt\leq C\left(\|(y,h,u)\|^{2}_{E} + \|(y,h,u)\|^{4}_{E} \right).
	\end{equation}
	Continuing, let $(y,h,u)\in E$
	$$
	\begin{array}{l}
	\displaystyle \iint_{Q}\hat{\rho}^{2}_{1}|\mathcal{A}_{1t}(y,h,u)|^{2}\, dxdt\\
				\noalign{\smallskip}\phantom{QQQQQW}
\displaystyle 
	=\iint_{Q}\hat{\rho}^{2}_{1}|y_{tt}-\nabla\cdot\left(a(\nabla y)\nabla y\right)_{t}+f'(y)y_{t}-u_{t}\chi_{\omega}|^{2}\, dxdt\\
				\noalign{\smallskip}\phantom{QQQQQW}
				\displaystyle \leq C \iint_{Q}\hat{\rho}^{2}_{1}|\mathcal{L}_{1}y_{t}-u_{t}\chi_{\omega}|^{2}\, dxdt+C \iint_{Q}\hat{\rho}^{2}_{1}|y_{t}|^{2}\, dxdt\\
							\noalign{\smallskip}\phantom{QQQQQtWA}
  \displaystyle + C\iint_{Q}\hat{\rho}^{2}_{1}|\nabla\cdot \left[(a(\nabla y)\nabla y)_{t}-a(0)\nabla y_{t}\right]|^{2}\, dxdt\\
				\noalign{\smallskip}\phantom{QQQQQWWWA}
=J_{1}+J_{2}+J_{3}.
	\end{array}
	$$
	From Proposition \ref{p-null-c-2} and Remark \ref{est-prop}, we have
	$$J_{1}\leq C\|(y,h,u)\|^{2}_{E}\,\,\ \ \mbox{and}\ \ \,\, J_{2}\leq C\|(y,h,u)\|^{2}_{E}, \,\,  \forall (y,h,u)\in E.$$
	On the other hand
	$$
	\begin{array}{l}
\displaystyle 	J_{3}\leq C\iint_{Q}\hat{\rho}^{2}_{1}|a(\nabla y)-a(0)|^{2}|\Delta y_{t}|^{2}\, dxdt+C\iint_{Q}\hat{\rho}^{2}_{1}|D^2_{ij} a( \nabla y)|^{2}|\nabla y_{t}|^{2}|\nabla y|^{2}|\Delta y|^{2}\, dxdt\\
				\noalign{\smallskip}\phantom{WA}
\displaystyle +C\iint_{Q} \hat{\rho}^{2}_{1}|D_i a( \nabla y)|^{2}|\Delta y|^{2}|\nabla y_{t}|^{2}\, dxdt+C\iint_{Q}\hat{\rho}^{2}_{1}|D_i a(\nabla y)|^{2}| \Delta y_{t}|^{2}|\nabla y|^{2}\, dxdt\\

				\noalign{\smallskip}\phantom{WA} 
				=J^{1}_{3}+J^{2}_{3}+J^{3}_{3}+J^{4}_{3}.
	\end{array}
	$$
	Now, let bound each term $J^{i}_{3}$, this is
	$$
	\begin{array}{ll}
						\phantom{WAGG} & 
\displaystyle J^{1}_{3}\leq C\iint_{Q}\hat{\rho}^{2}_{1}|\nabla y|^{2}\, dxdt \phantom{WAFFFDYYYgggeeeeeeegggYYdddddddeeehhhhhhhdddddddDDDDFFF} \\
						\noalign{\smallskip}
					\phantom{AGGHHA} & 	\phantom{J^{1}_{3}}
\displaystyle \leq C\left(\sup_{t\in [0,T]}(\hat{\rho}^{2}_{1}(t)\int_{\Omega}|\nabla y(t)|^{2}\, dx) \right)\left(\iint_{Q}\hat{\rho}^{2}_{5}| \Delta y_{t}|^{2}\, dxdt\right)
\\
								\noalign{\smallskip}
	\phantom{WHGGHA}  &	\phantom{J^{1}_{3}}
\displaystyle \leq C \|(y,h,u)\|^{4}_{E},	
	\end{array}
	$$
	$$	
	\begin{array}{ll}
							\phantom{WADD} &
	\displaystyle J^{2}_{3}\leq C\left(\sup_{t\in [0,T]}(\hat{\rho}^{2}_{5}(t)\int_{\Omega}|\nabla \Delta y(t)|^{2}\, dx)\right)\left(\iint_{Q}\hat{\rho}^{2}_{3}|\nabla y_{t}|^{2}\, dxdt\right) \phantom{WAhhhhhhhhhhhEeEdddddeeeeddddddeeeeeeeeeedE} \\
														\noalign{\smallskip}\phantom{WA} 
											\phantom{WHHA} & \phantom{J^{2}_{3}}
	\leq C \|(y,h,u)\|^{4}_{E},
	\end{array}
$$
	$$
	\begin{array}{ll}
					\phantom{WADD} 
 & \displaystyle J^{3}_{3}\leq C\left(\iint_{Q}\hat{\rho}^{2}_{4}|\Delta^{2} y|^{2}\, dxdt\right)\left(\sup_{t\in [0,T]}(\hat{\rho}^{2}_{3}(t)\int_{\Omega}|\nabla y_{t}(t)|^{2}\, dx)\right)\phantom{WAAAAAAAdddhhhhhhhhhdddeeddddeeeeeeeeeeeeddeeedd} \\
									\noalign{\smallskip}\phantom{WA} 
			\phantom{WDDA} &			\phantom{J^{3}_{3}} 
\displaystyle \leq	C \|(y,h,u)\|^{4}_{E},
	\end{array}
	$$
	$$
	\begin{array}{ll}
			\phantom{GGWA}& 
	\displaystyle J^{4}_{3}\leq C \left( \iint_{Q}\hat{\rho}^{2}_{3}|\Delta y_{t}|^{2}\, dxdt\right) \left(\sup_{t\in [0,T]}(\hat{\rho}^{2}_{5}(t)\int_{\Omega}|\nabla \Delta y(t)|^{2}\, dx) \right)\phantom{dddddddddddddddddddddddddddddddddddddddddddddWA}\\
														\noalign{\smallskip}\phantom{WA} 
							\phantom{WAHH} & 	\phantom{J^{4}_{3}}
 \leq C \|(y,h,u)\|^{2}_{E}.
	\end{array}
	$$
	Combining the four estimates of $J^{i}_{3},$ we have
	$$
	J_{3}\leq C\|(y,h,u)\|^{4}_{E}, \,\, \forall (y,h,u)\in E.
	$$
	and this conclude that
	$$
	\iint_{Q}\hat{\rho}^{2}_{1}|A_{1,t}(y,h,u)|^{2}\, dxdt\leq C\left( \|(y,h,u)\|^{2}_{E}+\|(y,h,u)\|^{4}_{E}\right).
	$$
	Now, let us to prove that $\displaystyle \hat{\rho}_{3}\Delta \mathcal{A}_{1}(y,h,u)\in L^{2}(Q), \,\, \forall\, (y,h,u)\in E.$ Indeed, we have
	$$
	\begin{array}{l}
	\displaystyle \iint_{Q}\hat{\rho}^{2}_{3}|\Delta \mathcal{A}_{1}(y,h,u)|^{2}\, dxdt\\
	 	\noalign{\smallskip}\phantom{WWQ}
	\displaystyle=\iint_{Q}\hat{\rho}^{2}_{3}|\Delta y_{t}-\Delta \nabla\cdot \left(a( \nabla y)\nabla y\right)+\Delta (f(y))-\Delta(u\chi_{\omega})|^{2}\, dxdt\\
	 	\noalign{\smallskip}\phantom{WWQ}
\displaystyle 
	\leq C\iint_{Q}\hat{\rho}^{2}_{3}|\mathcal{L}_{1}\Delta y-\Delta(u\chi_{\omega})|^{2}\, dxdt+C\iint_{Q}\hat{\rho}^{2}_{3}\left(|\nabla y|^{2}+| \Delta y|^{2}\right)\, dxdt\\
	 	\noalign{\smallskip}\phantom{WFWQ}
	 \displaystyle + C \iint_{Q}\hat{\rho}^{2}_{3}|\Delta \left[\nabla\cdot \left(a( \nabla y)\nabla y\right)-a(0)\Delta y\right]|^{2}\, dxdt\\
	 	\noalign{\smallskip}\phantom{WWQ}
	=K_{1}+K_{2}+K_{3}.
	\end{array}
	$$
	From definition of Banach space $E$ and Proposition \ref{p-null-c-2}, we have
	$$K_{1} \leq C\|(y,h,u)\|^{2}_{E}\,\, \mbox{and}\,\, K_{2}\leq C\|(y,h,u)\|^{2}_{E}.$$
	Now, let us bound the term $K_{3},$
	$$
	\begin{array}{l}
	\displaystyle K_{3}\leq C\sum_{i,j=1}^{N} \iint_{Q} \hat{\rho}^{2}_{3} \left| \left[\left(a( \nabla y)-a(0)\right) \Delta y\right]_{x_{i}x_{j}} \right|^{2}\, dxdt\\
	 	\noalign{\smallskip}\phantom{WWQ}
\displaystyle 	+C\sum_{i,j=1}^{N}\iint_{Q}\hat{\rho}^{2}_{3} \left|\left[D a(\nabla y)\cdot\nabla y \Delta y\right]_{x_{i}x_{j}} \right|^{2}\, dxdt\\
 	\noalign{\smallskip}\phantom{WQ}
	= K^{1}_{3}+K^{2}_{3}.
	\end{array}
	$$
	Let us denote by
	$$\tilde{K}^{1}_{3}=\left[\left(a(\nabla y)-a(0)\right)\Delta y\right]_{x_{i}x_{j}}\,\,\mbox{and} \,\, \tilde{K}^{2}_{3}=\left[D a(\nabla y)\cdot\nabla y \Delta y\right]_{x_{i}x_{j}}.$$
	Let us calculate each term in $\tilde{K}^{1}_{3}$ and $\tilde{K}^{2}_{3}$
	$$
	\begin{array}{l}
	\displaystyle \tilde{K}^{1}_{3}=D a(\nabla y)\cdot\nabla y_{x_{j}}\Delta y_{x_{i}}+\left(a(\nabla y)-a(0)\right) \Delta y_{x_{i}x_{j}}+\sum_{m,k=1}^{N}D^2_{mk} a( \nabla y)  y_{x_{m},x_{i}}y_{x_{k},x_{j}} \Delta y\\
	 	\noalign{\smallskip}\phantom{QQW}
	\displaystyle +D a(\nabla y)\Delta y_{x_{j}}\cdot\nabla y_{x_{i}}+D a(\nabla y)\Delta y\cdot\nabla y_{x_{i}x_{j}},
	\end{array}
	$$
	and
	$$
	\begin{array}{l}
	\tilde{K}^{2}_{3}=\left(D a(\nabla y)\right)_{x_{i},x_{j}}\cdot\nabla y\Delta y+ \left(Da(\nabla y)\right)_{x_{i}}\cdot\nabla y_{x_{j}} \Delta y+\left(D a( \nabla y)\right)_{x_{i}}\Delta y_{x_{j}}\\
		 	\noalign{\smallskip}\phantom{QQW}
	\displaystyle 
	+\left(D a(\nabla y)\right)_{x_{j}}\cdot\nabla y_{x_{i}}\Delta y+ D a(\nabla y)\cdot\nabla y_{x_{i}x_{j}}\Delta y+D a(\nabla y)\cdot\nabla y_{x_{i}}\Delta y_{x_{j}}\\
		 	\noalign{\smallskip}\phantom{QQW}
	\displaystyle + \left(D a(\nabla y)\right)_{x_{j}}\cdot\nabla y \Delta y_{x_{i}}+ D a(\nabla y)\cdot\nabla y_{x_{j}}\Delta y_{x_{i}}+ D a( \nabla y)\cdot\nabla y\Delta y_{x_{i}x_{j}}.
	\end{array}
	$$
	Due it, we have that
	$$
	\begin{array}{l}
\displaystyle 	K^{1}_{3}=\iint_{Q}\hat{\rho}^{2}_{3}|\tilde{K}^{1}_{3}|^{2}\, dxdt \leq K^{1}_{1,3}+\cdots +K^{1}_{5,3},
	\end{array}
	$$
	and
	$$
	\begin{array}{l}
	\displaystyle K^{2}_{3}=\iint_{Q}\hat{\rho}^{2}_{3}|\tilde{K}^{2}_{3}|^{2}\, dxdt \leq K^{2}_{1,3}+\cdots +K^{2}_{9,3}.
	\end{array}
	$$
	Using Proposition \ref{p-null-c-2} and Remark \ref{est-prop}, we have the following estimates
	$$
	\begin{array}{l}
\displaystyle 	\sum_{i=1}^{5}K^{1}_{i,3}+\sum_{j=1}^{9}K^{2}_{j,3}\leq C\left(\|(y,h,u)\|^{2}_{E} + \|(y,h,u)\|^{4}_{E} + \|(y,h,u)\|^{6}_{E}\right),
	\end{array}
	$$
	and therefore,
	$$\displaystyle K_{3}\leq C\left(\|(y,h,u)\|^{2}_{E} + \|(y,h,u)\|^{4}_{E} + \|(y,h,u)\|^{6}_{E}\right).
$$
	For all this, we conclude that
	$$\hat{\rho}_{3}\Delta \mathcal{A}_{1}(y,h,u)\in L^{2}(Q).$$
	
	\item [ii)] $\displaystyle \mathcal{A}_{2}(y,h,u)\in F_{2}, \,\, \forall\, (y,h,u)\in E$.
	
	To prove this, we use arguments as above, this is
	$$
	\begin{array}{l}
\displaystyle 	\iint_{Q}\rho^{2}|\mathcal{A}_{2}(y,h,u)|^{2}\, dxdt\leq	\iint_{Q}\rho^{2}|\mathcal{L}_{2}h-y\chi_{\omega}|^{2}\, dxdt+C\iint_{Q}\rho^{2}|h|^{2}\,dxdt\\
 	\noalign{\smallskip}\phantom{QQgggAAASSSggggW}
\displaystyle+C\iint_{Q} \rho^{2}|\nabla\cdot \left[Da(\nabla y)(\nabla y\cdot \nabla h)+a(\nabla y)\nabla h-a(0)\nabla h\right] |^{2}\, dxdt\\
 	\noalign{\smallskip}\phantom{QQWAAAAAAAWWQ}
=L_{1}+L_{2}+L_{3}.
	\end{array}
	$$
	Similarly as in the estimates of item $i)$ we will use Proposition \ref{p-null-c-2} and Remark \ref{est-prop}, such that
	$$
	L_{j}\leq C\left(\|(y,h,u)\|^{2}_{E} + \|(y,h,u)\|^{4}_{E} + \|(y,h,u)\|^{6}_{E}\right),\,\ j=1,2,3.
	$$
	Then $\displaystyle \mathcal{A}_{2}(y,h,u)\in F_{2}, \,\, \forall \, (y,h,u)\in E$ holds.
\end{description}
Therefore, $\mathcal{A}$ is well defined.
	
	Furthermore, using similar arguments, it is easy to check that $\mathcal{A}$ is continuous.
\end{proof}
\begin{lemma}\label{lem2}
The mapping $\mathcal{A}: E \rightarrow F$ is continuously differentiable.
\end{lemma}

\begin{proof}
Let us first prove that $\mathcal{A}$ is G\^ateaux-differentiable at any  $(y,h,u)\in E$ and calculate the $G$-derivate $\mathcal{A}'(y,h,u).$

For this purpose let us introduce the linear mapping 
$$\displaystyle D\mathcal{A}:E\mapsto F,\,\, \mbox{ with }\,\, \displaystyle D\mathcal{A}(\tilde{y},\tilde{h},\tilde{u})=\left(D\mathcal{A}_{1}(\tilde{y},\tilde{h},\tilde{u}),D\mathcal{A}_{2}(\tilde{y},\tilde{h},\tilde{u})\right),$$
$$
\left\{\begin{array}{l}
D\mathcal{A}_{1}(\tilde{y},\tilde{h},\tilde{u})=\tilde{y}_{t}-\nabla\cdot \left[\left(D a( \nabla y)\cdot\nabla \tilde{y}\right)\nabla y-a(\nabla y)\nabla \tilde{y} \right]-\tilde{u}\chi_{\omega}+f'(y)\tilde{y},\\
 	\noalign{\smallskip}
D\mathcal{A}_{2}(\tilde{y},\tilde{h},\tilde{u})=\tilde{h}_{t}-\nabla\cdot\left[ (D^2 a(\nabla y)\cdot \nabla \tilde{y})(\nabla\tilde{y}\cdot\nabla h)+Da(\nabla y)(\nabla \tilde{y}\cdot \nabla h)\right.\\
	\noalign{\smallskip}\phantom{QQWWWWWQ}
+ D a(\nabla y)(\nabla y\cdot\nabla h) \left. +(D a(\nabla y)\cdot\nabla \tilde{y}) \nabla h+a(\nabla y)\nabla \tilde{h}\right]\\
	\noalign{\smallskip}\phantom{QQWWWWWQ}
\displaystyle + f''(y)\tilde{y}h+f'(y)\tilde{h}-\tilde{y}\chi_{\mathcal{O}}.
\end{array}\right.
$$
To prove that $D\mathcal{A}$ is the $G$-derivative of $\mathcal{A}$ in $(y,h,u),$ we will prove that
\begin{equation}\label{D-a}
\frac{1}{\lambda}\left[\mathcal{A}_{i}\left((y,h,u)+\lambda(\tilde{y},\tilde{h},\tilde{u})\right)-\mathcal{A}_{i}(y,h,u)\right]\to D\mathcal{A}_{i}(\tilde{y},\tilde{h},\tilde{u}),
\end{equation}
strongly in $F_{i}$ for $i=1,2$ as $\lambda \to 0.$ Indeed, we have
$$
\begin{array}{l}
\displaystyle \frac{1}{\lambda}\left[\mathcal{A}_{1}\left((y,h,u)+\lambda(\tilde{y},\tilde{h},\tilde{u}) \right)-\mathcal{A}_{1}(y,h,u) \right]-D\mathcal{A}_{1}(\tilde{y},\tilde{h},\tilde{u})\\

 	\noalign{\smallskip}\phantom{QQDDDDDDWW}
\displaystyle=-\nabla\cdot\left[\left(\frac{a( \nabla y+\lambda\nabla \tilde{y} )-a(\nabla y)}{\lambda}-Da(\nabla y)\cdot\nabla \tilde{y}\right)\nabla y \right]\\

 	\noalign{\smallskip}\phantom{QQWWWDDWWQ}
\displaystyle-\nabla\cdot\left[ \left( a(\nabla y+\lambda \nabla \tilde{y})-a(\nabla y)\right)\nabla \tilde{y}\right]\\

 	\noalign{\smallskip}\phantom{QQWWWGGWWQ}
\displaystyle=\tilde{J}_{1}+\tilde{J}_{2}.
\end{array}
$$
Similarly as in the prove of Lemma \ref{lem1}, we have
$$
\|\tilde{J}_{1}\|_{F_{1}}\to 0\,\,\ \mbox{as}\,\, \lambda \to 0,$$
and 
$$
\|\tilde{J}_{2}\|_{F_{1}}\to 0\,\,\ \mbox{as}\,\, \lambda \to 0.
$$
Then \eqref{D-a} holds for $i=1.$

Now, let us calculate the following term
$$
\begin{array}{l}
\displaystyle \frac{1}{\lambda}\left[\mathcal{A}_{2}\left((y,h,u)+\lambda(\tilde{y},\tilde{h},\tilde{u})\right)-\mathcal{A}_{2}\left(y,h,u\right) \right]-D\mathcal{A}_{2}(\tilde{y},\tilde{h},\tilde{u})\\
 		\noalign{\smallskip}\phantom{SAAAA} \displaystyle = -\nabla\cdot \left[\left(\frac{Da( \nabla y+\lambda \nabla \tilde{y})-Da(\nabla y)}{\lambda}\nabla \tilde{y}- D^2 a(\nabla y)\cdot\nabla \tilde{y} \right)(\nabla y \cdot\nabla h) \right]\\
			\noalign{\smallskip}\phantom{SAAAAA}
\displaystyle 			- \nabla\cdot \left[\left(\frac{a(\nabla y+\lambda \nabla \tilde{y})-a(\nabla y)}{\lambda}-Da(\nabla y)\cdot\nabla\tilde{y}\right)\nabla h\right]\\
			\noalign{\smallskip}\phantom{SAAAAA}
		\displaystyle -\nabla\cdot\left[\left(Da(\nabla y+\lambda \nabla y)-Da(\nabla y) \right)\nabla \tilde{y}\cdot\nabla h \right]+\nabla\cdot\left[\left(Da(\nabla+ \lambda \nabla \tilde{y})-Da(\nabla y)\right)\nabla y\cdot\nabla \tilde{h}\right]\\
	\noalign{\smallskip}\phantom{SAAAAA}
\displaystyle -\nabla\cdot \left[\left(a(\nabla y+\lambda \nabla \tilde{y})-a(\nabla y)\right)\nabla \tilde{h}\right]-\lambda \nabla\cdot\left[Da(\nabla y+\lambda \nabla \tilde{y})(\nabla \tilde{y}\cdot\nabla \tilde{h})\right]\\
	\noalign{\smallskip}\phantom{SAAAAA}
	\displaystyle +\left(f'(y+\lambda\tilde{y})-f'(y) \right)\tilde{h}+\left(\frac{f'(y+\lambda\tilde{y})-f'(y)}{\lambda}-f''(y)\tilde{y} \right)h\\
		\noalign{\smallskip}\phantom{SAAAASA}
		=\tilde{N}_{1}+\cdots+\tilde{N}_{8}.
\end{array}
$$
Using Proposition \ref{p-null-c-2} and Remark \ref{est-prop}, we have that
$$\|\tilde{N}_{j}\|_{F_{2}}\to 0 \,\,\mbox{as}\,\, \lambda\to 0, \, \mbox{for}\, j=1,...,8.$$
Then \eqref{D-a} holds for $i=2.$

Therefore, $\mathcal{A}$ is G\^ateaux-differentiable.

Now, let us check that $\mathcal{A}\in C^{1}(E;F)$ with $\mathcal{A}'(y,h,u)=D_{G}\mathcal{A}(y,h,u)$ i.e
$$
\mathcal{A}'(y,h,u)(\tilde{y},\tilde{h},\tilde{u})=D_{G}\mathcal{A}(y,h,u)(\tilde{y},\tilde{h},\tilde{u}).
$$
But this last equality is equivalent to prove that there exists $\epsilon_{n}(y,h,u)$ such that
\begin{equation}\label{D-iq}
	\|\left(D_{G}\mathcal{A}(y^{n},h^{n},u^{n})-D_{G}\mathcal{A}(y,h,u)\right)(\tilde{y},\tilde{h},\tilde{u})\|^{2}_{F} \leq \epsilon_{n} \|(\tilde{y},\tilde{h},\tilde{u})\|^{2}_{E},
\end{equation}
for all $(\tilde{y},\tilde{h},\tilde{u})\in E$ and $\displaystyle \lim_{n\to \infty}\epsilon_{n}=0.$

Let us to prove \eqref{D-iq}, 
$$
\begin{array}{l}
D_{G}\mathcal{A}_{1}(y^{n},h^{n},u^{n})(\tilde{y},\tilde{h},\tilde{u})-D_{G}\mathcal{A}_{1}(y,h,u)(\tilde{y},\tilde{h},\tilde{u})\\
 		\noalign{\smallskip}\phantom{SAAADDD}
\displaystyle =-\nabla\cdot \left[\left(Da(\nabla y)\cdot\nabla \tilde{y}\right)(\nabla y^{n}-\nabla y) \right]-\nabla\cdot \left[\left(\left(Da(\nabla y^{n})-Da(\nabla y)\right)\cdot\nabla \tilde{y}\right)\nabla y \right]\\
 		\noalign{\smallskip}\phantom{SAAAADDD}
\displaystyle +\nabla\cdot\left[\left(a(\nabla y^{n})-a(\nabla y)\right)\nabla \tilde{y}\right]+ \left(f'(y^{n})-f'(y)\right)\tilde{y}\\
 		\noalign{\smallskip}\phantom{SAAADDD}
=\tilde{O}_{1,n}+\tilde{O}_{2,n}+\tilde{O}_{3,n}+\tilde{O}_{4,n}.
\end{array}
$$
By properties of the functions $a(\cdot)$, $f(\cdot)$, Proposition \ref{p-null-c-2} and Remark \ref{est-prop}, we have
$$
\displaystyle \|\tilde{O}_{j,n}\|_{F_{1}}\leq \epsilon^{1}_{j,n} \|(\tilde{y},\tilde{h},\tilde{u})\|_{E},\, \mbox{with}\, \, \lim_{n\to \infty}\epsilon^{1}_{j,n}=0,\, j=1,..,4.
$$
All this implies that
\begin{equation}\label{D-iq1}
	\|\left(D_{G}\mathcal{A}_{1}(y^{n},h^{n},u^{n})-D_{G}\mathcal{A}_{1}(y,h,u)\right)(\tilde{y},\tilde{h},\tilde{u})\|_{F_{1}} \leq \left(\sum_{j=1}^{4}\epsilon_{j,n}^1 \right)\|(\tilde{y},\tilde{h},\tilde{u})\|_{E}.
\end{equation}
Continuing, let us calculate a similar expression for $\mathcal{A}_{2}$,
$$
\begin{array}{l}
\left(D_{G}\mathcal{A}_{2}(y^{n},h^{n},u^{n})-D_{G}\mathcal{A}_{2}(y,h,u)\right)(\tilde{y},\tilde{h},\tilde{u})\\
 		\noalign{\smallskip}\phantom{DD}
=-\nabla\cdot\left[\left(\left(D^2 a(\nabla y^{n})-D^2 a( \nabla y)\right)\cdot\nabla \tilde{y}\right)(\nabla y^{n}\cdot\nabla h^{n}) \right]\\
 		\noalign{\smallskip}\phantom{SAA}
-\nabla\cdot\left[D^2 a( \nabla y)\cdot\nabla \tilde{y}\left(\nabla y^{n}-\nabla y\right)\cdot\nabla h^{n} \right]-\nabla\cdot\left[\left(Da(\nabla y^{n})-Da(\nabla y)\right)(\nabla \tilde{y}\cdot\nabla h^{n}) \right]\\
 		\noalign{\smallskip}\phantom{SAA}
        -\nabla\cdot\left[ \left(D^2 a( \nabla y)\cdot\nabla\tilde{y}\right)(\nabla y\cdot(\nabla h^{n}-\nabla h))\right]-\nabla\cdot\left[\left(Da(\nabla y^{n})-Da(\nabla y)\right)\left(\nabla y^{n}\cdot\nabla\tilde{h}\right)\right]\\
 		\noalign{\smallskip}\phantom{SAA}
-\nabla\cdot\left[Da(\nabla y)\left(\left(\nabla y^{n}-\nabla y\right)\cdot\nabla\tilde{h}\right)\right]
 		-\nabla \cdot\left[\left(Da(\nabla y^{n})-Da(\nabla y)\right)\nabla \tilde{y}\cdot \nabla h^{n} \right]\\
 		\noalign{\smallskip}\phantom{SAA}
        -\nabla \cdot\left[Da(\nabla y)\left(\nabla \tilde{y}\cdot\nabla (h^{n}-h)\right) \right]
-\nabla\cdot\left[\left(a(\nabla y^{n})-a(\nabla y)\right)\nabla \tilde{h}\right]\\
 		\noalign{\smallskip}\phantom{SAA}
        -\nabla \cdot\left[\left(Da(\nabla y)\cdot\nabla \tilde{y}\right)\nabla (h^{n}-h) \right] +\left(f''(y^{n})-f''(y)\right)\tilde{y}h^{n}\\
 		\noalign{\smallskip}\phantom{SAA}
        +f''(y)\tilde{y}\left(h^{n}-h\right)+\left(f'(y^{n})-f'(y)\right)\tilde{h}\\
 		\noalign{\smallskip}\phantom{AA}
=\tilde{P}_{1,n}+\cdots+\tilde{P}_{13,n}.
\end{array}
$$
Thanks to Proposition \ref{p-null-c-2} and Remark \ref{est-prop}, we have
$$
\displaystyle \|\tilde{P}_{j,n}\|_{F_{2}} \leq \epsilon^{2}_{j,n} \|(\tilde{y},\tilde{h},\tilde{u})\|_{E}, \,\,  j=1,...,13\,\,\mbox{with}\,\, \lim_{n\to \infty}\epsilon^{2}_{j,n}=0.
$$
Then
\begin{equation}\label{D-iq2}
	\|\left(D_{G}\mathcal{A}_{2}(y^{n},h^{n},u^{n})-D_{G}\mathcal{A}_{2}(y,h,u)\right)(\tilde{y},\tilde{h},\tilde{u})\|_{F_{2}}\leq \left(\sum_{j=1}^{13}\epsilon^{2}_{j,n} \right) \|(\tilde{y},\tilde{h},\tilde{u})\|_{E}.
\end{equation}
From \eqref{D-iq1}-\eqref{D-iq2}, we have that \eqref{D-iq} holds.
 
This ends the proof.
\end{proof}

\begin{lemma}\label{lem3}
Let $\mathcal{A}$ be the mapping defined by $(\ref{def-A})$. Then $\mathcal{A}'(0,0,0)$ is onto. 
\end{lemma}

\begin{proof}
Let us fix $(g_{1},g_{2}) \in F$. From Proposition $\ref{p-null-c-1}$, we know that there exists $(y,h,u)$ satisfying $\eqref{linear-1}$, \eqref{eq-null-1}, \eqref{eq-null-2} and \eqref{est-2}. Also, on other hand we have from Proposition \ref{p-null-c-2} that  $y$ satisfies \eqref{eq-null-3} and \eqref{eq-null-4}. Therefore this $(y,h,u)$ found belong the space $E$ and consequently $\mathcal{A}'(0,0,0)(y,h,u)=(g_{1},g_{2})$.

This end the proof.
\end{proof}
In accordance with Lemmas $\ref{lem1}$, $\ref{lem2}$ and $\ref{lem3}$, we can apply Liusternik's Theorem $($Theorem $\ref{Liu-1})$ and deduce that, there exists $\epsilon>0$, a mapping $W: B_\epsilon(0)\subset F \to E$ such that 
$$W(g_{1},g_{2}) \in B_r(0) \ \ \text{and} \ \ \mathcal{A}(W(g_{1},g_{2}))=(g_{1},g_{2}), \, \ \ \forall (g_{1},g_{2}) \in B_\epsilon(0).$$
In particular, taking  $\epsilon_{0}<\epsilon$  sufficiently small  and  $(\xi,0) \in B_{\epsilon_{0}}(0)$, there exists $(y,h,u)\in E$ such that  $(y,h,u)=W(\xi,0) \in E$, we have 
$$\mathcal{A}((y, h, u))=(\xi,0),$$
thus, we prove that $(\ref{EC1})$ is null locally controllable at time $t=0$.

\subsection{Proof of Theorem~\ref{main-t}}

From   the  locally null controllability of  system \eqref{eq-char}, we have  that  there  exist $\delta>0$  sufficiently small  and  $\tilde{M}$ such  that  if  
\begin{equation}\label{T-ESTIMATE}
    \displaystyle \left\|e^{\frac{\tilde{M}}{t}}\ \xi \right\|_{X_{0}}<\delta,
\end{equation}
then  one  can find  a  control  $u$  with $\mathrm{supp}\ u \subset \omega\times[0,T]$  such that  the state  associated $(y,h)$  satisfies  $\displaystyle h(x,0)=0$ in $\Omega$. Thanks to \eqref{eq-null-1}, \eqref{T-ESTIMATE} and Proposition \ref{char-e},  this  implies  that  the   control  $u$  insensitizes  the  functional $\Phi$  in the sense of Definition \ref{def-ins}.

This  ends  the  proof.


\section{Some Additional Comments and Questions}\label{sec5}

\setcounter{equation}{0}

\subsection{Insensitizing controls for a quasi-linear parabolic equation with diffusion depending on gradient of the state in one dimension}
When $N=1$, we have that  the system \eqref{EC1} can be rewritten the following way
\begin{equation}\label{EC1-N1}
\left\{\begin{array}{lll}
  y_t -  \left(a( y_{x})  y_{x}\right)_x+f(y)= \xi+u\chi_{\omega}  & \text{in}& \ I\times (0,T),\\
  y(0,t)= 0,\, y(1,t)=0 &  \text{on} &\ (0,T),\\
  y(x,0) = y_{0}(x)+\tau \hat{y}_{0}(x)&  \text{in} &\ I,\\
\end{array}\right.
\end{equation}
here, $I=(0,1),\,\, a\in C^{2}(\mathbb{R})$ and $f\in C^{2}(\mathbb{R})$ with $f(0)=0$ satisfying
\begin{equation}\label{hy-a1}
\begin{array}{l}
a_{0}\leq a(r)\leq a_{1},\, |a'(r)|+|a''(r)|+|f'(r)|+f''(r)|\leq M, \,\, \forall \, r\in \mathbb{R}.
\end{array}
\end{equation}

We have the following result:

\begin{theorem}\label{main-1d}
Assume $\omega \cap\mathcal{O}\neq \emptyset$ and $y_{0}=0.$ Then, there exist two positive constants $\tilde{M}$ and $\delta$ depending only on $I,\, T,\, M,\, a_{0}$ and $a_{1}$ such that for any $\xi\in H^{1}(0,T;L^{2}(I))$ satisfying
$$
\left\|e^{\frac{\tilde{M}}{t}}\ \xi \right\|_{H^{1}(0,T;L^{2}(\Omega))} \leq\delta,
$$
one can find a control function $u\in H^{1}(0,T;L^{2}(\omega)),$ which insensitizes the functional $\Phi$. 
\end{theorem}
The proof can be argued as in Section \ref{sec3}. Indeed, the insensitizing problem is equivalent to null controllability for the following system
\begin{equation}\label{equi-C1}
\left\{\begin{array}{lll}
y_{t}- \left(a(y_{x}) y_{x}\right)_{x}+f(y)=\xi+u\chi_{\omega} & \mbox{in} &I\times (0,T),\\
\noalign{\smallskip}
-h_{t}-\left[\left(a'( y_{x}) y_{x}\right) h_{x}+a(y_{x}) h_{x}\right]_{x} +f'(y)h=y\chi_{\mathcal{O}} &\mbox{in} &I\times (0,T),\\
\noalign{\smallskip}
y(0,t)=y(1,t)=0,\, h(0,t)=h(1,t)=0  &\mbox{on} & (0,T),\\
\noalign{\smallskip}
y(x,0)=0,\, h(x,T)=0 &\mbox{in} & I,
\end{array}\right.
\end{equation}
and to guarantee that the system \eqref{equi-C1} be null controllable (for instance, see \cite{Mig}), we consider the linearized system
\begin{equation}\label{equi-L1}
\left\{\begin{array}{lll}
y_{t}-a(0) y_{xx}+Ay=u\chi_{\omega}+F_{1} & \mbox{in} &I\times (0,T),\\
\noalign{\smallskip}
-h_{t}-a(0) h_{xx}+Ah=y\chi_{\mathcal{O}}+F_{2} &\mbox{in} & I\times(0,T),\\
\noalign{\smallskip}
y(0,t)=y(1,t)=0,\, h(0,t)=h(1,t)=0  &\mbox{on} & (0,T),\\
\noalign{\smallskip}
y(x,0)=0,\, h(x,T)=0 &\mbox{in} & I.
\end{array}\right.
\end{equation}
Arguing as in Section \ref{sec4}, we can establish  the null controllability for \eqref{equi-L1} with a state-control $(y,h,u)$ satisfying \eqref{eq-null-2} and \eqref{eq-null-3}. Then, we can introduce the Banach space
$$
\begin{array}{l}
\tilde{E}=\Big\{ (y,h,u):\, \rho y,\, \rho h,\, y_{x},\, y_{xx}\, \in L^{2}(Q);\, \rho u,\, \hat{\rho}_{1}u_{t}\in L^{2}(\omega\times (0,T)),\Big.\\
\noalign{\smallskip}\phantom{QQQ}
\displaystyle \, \rho (\mathcal{L}_{1}y-u\chi_{\omega}),\, \hat{\rho}_{1}(\mathcal{L}_{1}y_{t}-u_{t}\chi_{\omega}),\, \rho (\mathcal{L}_{2}h-y\chi_{\mathcal{O}})\in L^{2}(Q),\\
\noalign{\smallskip}\phantom{QQQ}
\displaystyle
\, \Big. \, y|_{\Sigma}=0,\, h|_{\Sigma}=0,\, h(T)= 0\Big\},
\end{array}
$$
and the norm in $\tilde{E}$ is 
$$
\begin{array}{l}
	\|(y,h,u)\|^{2}_{\tilde{E}}\\
 		\noalign{\smallskip}\phantom{SAAAD}
\displaystyle = \|\rho y\|^{2}_{L^{2}(Q)} + \|\rho h\|^{2}_{L^{2}(Q)} + \|\hat{\rho}u\|^{2}_{L^{2}(Q)}+ \|\rho\left(\mathcal{L}_{1}y-u\chi_{\omega}\right)\|^{2}_{L^{2}(Q)}\\
 		\noalign{\smallskip}\phantom{SAAAD}
\displaystyle + \|\hat{\rho}\left(\mathcal{L}_{1}y_{t}-u_{t}\chi_{\omega}\right)\|^{2}_{L^{2}(Q)} + \|\hat{\rho}\left(\mathcal{L}_{2}h-y\chi_{\cal O}\right)\|^{2}_{L^{2}(Q)}.
\end{array}
$$
Also, let us define the following  Banach space
$$
\begin{array}{l}
\tilde{F}_{1}=\Big\{ w:\, \rho w,\, \hat{\rho}_{1}w_{t}\in L^{2}(Q)\Big\},\\
 		\noalign{\smallskip}
\tilde{F}_{2}=\Big\{ w:\, \rho w\in L^{2}(Q)\Big\},\\
 		\noalign{\smallskip}
\tilde{F}=\tilde{F}_{1}\times \tilde{F}_{2},
\end{array}
$$
with its norm
$$\|(f,g)\|^{2}_{\tilde{F}}=\|\rho f\|^{2}_{L^{2}(Q)}+\|\hat{\rho}_1 f_{t}\|^{2}_{L^{2}(Q)}+\|\rho g\|^{2}_{L^{2}(Q)},$$
and the nonlinear mapping $\displaystyle \mathcal{A}:\tilde{E}\mapsto \tilde{F},$ with $\displaystyle \tilde{A}(y,h,u)=\left(\mathcal{A}_{1}(y,h,u),\mathcal{A}_{2}(y,h,u) \right)$ where
\begin{equation}\label{equiv-A}
\left\{\begin{array}{l}
\mathcal{\tilde{A}}_{1}(y,h,u)=y_{t}-\left(a( y_{x})y_{x}\right)_{x}+f(y)-u\chi_{\omega},\phantom{QQQQQqq}\\
\noalign{\smallskip}
\mathcal{\tilde{A}}_{2}(y,h,u)=-h_{t}-\left[\left(a'( y_{x}) y_{x}+a( y_{x})\right) h_{x} \right]_{x}+f'(y)h-y\chi_{\mathcal{O}}.
\end{array}\right.
\end{equation}
Then, we can show that Lemmas \ref{lem1}, \ref{lem2} and \ref{lem3} holds. All this is due  to similar arguments as in the Section \ref{sec4} and  thanks to $H^{1}(I)\hookrightarrow L^{\infty}(I)$. Therefore, $\mathcal{\tilde{A}}$ is $C^{1}$ with $\mathcal{\tilde{A}}'(0,0,0)$ onto. In particular  the equation 
$$
\mathcal{\tilde{A}}(y,h,u)=(\xi,0),\, (y,h,u)\in \tilde{E},
$$
is solvable and \eqref{equi-C1} is locally null controllable.

\begin{remark}
In the unidimensional case, we can prove the same result of Theorem \ref{main-t} but assuming less regularity in the initial condition $\xi$. It  is true, because  when $N=1$  we have  $H^{1}(I)\hookrightarrow L^{\infty}(I)$    and that makes it easy to obtain the inequality
\begin{equation}\label{A-eq1d}
	\|\left[\left(a(y_{x})-a(0)\right)y_{x}\right]_{x}\|_{\tilde{F}_{1}} \leq C \|(y,h,u)\|_{\tilde{E}},
\end{equation}
where  $C$  is  a  positive constant  and,
\begin{equation}\label{A-eq1d2}
\tilde{\cal A}\,\,\,\mbox{  satisfies  the  hypothesis of Liusternik's Theorem.}
\end{equation}
 The estimate \eqref{A-eq1d} is very important to prove Theorem \ref{main-1d}.
\end{remark}

\subsection{Insensitizing controls for the system \eqref{EC1} when  $N\geq 4$ }

When $N\geq 4,$ if we apply the same techniques, we need estimate the following term
$$\left\|\nabla \cdot \Big[\left(a( \nabla y)-a(0)\right)\nabla y \Big]\right\|_{F_{1}}\leq C \|(y,h,u)\|_{E},$$ 
for the function $\mathcal{A}$ is well defined  and obtain the null controllability for the system \eqref{EC1}, but it is very difficult, because   in the  prove  to Theorem \ref{main-t}  we use  the immersion $\displaystyle H^{2}(\Omega)\hookrightarrow L^{\infty}(\Omega)$  and  this  only  is  valid  when $N=1,\, 2$ or $3$. So, this question is open.

\subsection{Numerical Results}

The  strategy we have  followed  in the  proof  of  Theorem \ref{main-t} opens  the  possibility  to solve  numerically  the insensitizing control problem of $\Phi$ for \eqref{EC1}. Indeed, it is  completely  natural  to introduce  algorithms  of  the  quasi-Newton  kind  to compute   (an  approximation  in time  and  space  of)  a solution  to the null controllability of \eqref{eq-char}. And so solve the  insensitizing control problem of $\Phi$ for \eqref{EC1}. 

These  ideas  have  already  been  applied  in \cite{Santa1} and \cite{Irene-j}  (among  other  references) in  the context  of the  controllability  of  other  equations  and  systems.


\backmatter

\begin{appendices}

\section{Well-possessedness  for  the system \eqref{EC1}}\label{Apex_1}
\setcounter{equation}{0}
Consider the following system:
\begin{equation}\label{A1}
\left\{
\begin{array}{lll}
y_{t}-\nabla\cdot\left( a(\nabla y)\nabla y\right)+f(y)=g & \mbox{in} &Q,\\
	\noalign{\smallskip} 
y(x,t)=0 & \mbox{on} & \Sigma,\\
	\noalign{\smallskip} 
y(x,0)=y_{0} &\mbox{in} & \Omega. 
\end{array}
\right.
\end{equation}
In order to prove the existence of solution for the  system \eqref{EC1}, we will first study the well-possessedness of system \eqref{A1}. We have the following.
\begin{lemma}\label{well-p1}
There exist $r>0$ such that  for each $y_{0}\in H^{3}(\Omega)\cap H^{1}_{0}(\Omega)$ with $\Delta y_0 \in H_0^1(\Omega)$ and $g\in X_{0}$ satisfying
$$
\|y_{0}\|_{H^{3}(\Omega)}+\|g\|_{X_{0}}\leq r,
$$
the problem \eqref{A1} has a unique solution $y$ in $X_{1}$ satisfying
\begin{equation}\label{A2}
	\|y\|_{X_{1}}\leq C \left( \|y_{0}\|_{H^{3}(\Omega)}+ \|g\|_{X_{0}}\right),
\end{equation}
where $C$ is a constant only depending  on $\Omega, T, M, a_0$ and $a_1.$
\end{lemma}
\begin{proof}
 We  employ Faedo-Galerkin method with  the Hilbert basis from $H^{1}_{0}(\Omega)$, given by  eigenvectors $(w_{j})$ of the spectral problem $\left((w_{j},v)\right)=\lambda_{j} (w_{j},v),$ for all $v\in V=H^{2}(\Omega)\cap H^{1}_{0}(\Omega)$ and $j=1,2,3,...$ We  represent  by $V_{m}$ the subspace of $V$ generated by vectors $\{w_1,w_2,...,w_m\}.$ We propose the following approximate problem:
\begin{equation}\label{A3}
\left\{
\begin{array}{l}
\left(y'_{m},v\right)+\left(a(\nabla y_{m})\nabla y_{m},\nabla v\right)+(f(y_{m}),v)=\left(g,v\right), \,\, \forall \, v\in V_m,\\
	\noalign{\smallskip} 
y_{m}(0)=y_{0m} \to y_{0}\,\, \mbox{in}\,\, H^{3}(\Omega)\cap H^{1}_{0}(\Omega).
\end{array}
\right.
\end{equation}
The existence and uniqueness of (local in time) solution to the \eqref{A3} is ensure by classical ODE theory. The following estimates show that in fact, they are defined for all $t$. We can get uniform estimates of $y_{m}$ in the usual way:

{\bf Estimate I:} Taking $v=y_{m}(t)$ in \eqref{A3}, we deduce that
\begin{equation}\label{A-eq1}
	\displaystyle \frac{1}{2}\frac{d}{dt}\|y_m(t)\|^{2}+\frac{a_0}{2}\| \nabla y_m(t)\|^{2}\leq \tilde{C}_{1}\left(\|y_{m}(t)\|^{2}+\|g(t)\|^{2}\right),
\end{equation}
and 
\begin{equation}\label{A-eq2}
	\|y_{m}(t)\|^{2}+\|\nabla y_{m}\|^{2}_{L^{2}(Q)}\leq \tilde{C}_{2}\left(\|y_0\|^{2}+\|g\|^{2}_{L^{2}(Q)} \right).
\end{equation}
In the sequel, the symbol $\tilde{C}_{k}$ is a constant that only depend of $a_0,\, a_1,\, M$ and $\Omega$ for $k=1,...,17.$

Notice that the term of side right in \eqref{A-eq2} don't dependent of $m$ and due it, we can extend the solution $(y_m)$ to the interval $[0,T].$

{\bf Estimate II:} Taking $v=- \Delta y_{m}(t)$ in \eqref{A3}, we see that
$$
\begin{array}{l}
	\displaystyle \frac{1}{2}\frac{d}{dt}\|\nabla y_{m}\|^{2} + \int_{\Omega}a(\nabla y_{m})|\Delta y_{m}|^{2}\,dx + \sum_{j=1}^{N}\int_{\Omega}Da(\nabla y_{m})\cdot\frac{\partial}{\partial x_{j}}(\nabla y_{m})\frac{\partial y_{m}}{\partial x_{j}}\Delta y_{m}dx\\
	\noalign{\smallskip} \phantom{\frac{1}{2}||\nabla y'_{m}||}
\displaystyle - \int_{\Omega}f(y_{m})\Delta y_{m}\,dx=-\int_{\Omega}g \Delta y_{m}\,dx.
\end{array}
$$
Passing  the last two terms  from the left  side  to the  right  side, and by absolute  value inequalities,
$$
\begin{array}{l}
	\displaystyle \frac{1}{2}\frac{d}{dt}\|\nabla y_{m}\|^{2} + \int_{\Omega}a(\nabla y_{m})|\Delta y_{m}|^{2}\,dx\\
	\noalign{\smallskip} \phantom{\frac{1}{2}||\nabla y'_{m}||}
\displaystyle \leq \left|\sum_{j=1}^{N}\int_{\Omega}Da(\nabla y_{m})\cdot\frac{\partial}{\partial x_{j}}(\nabla y_{m})\frac{\partial y_{m}}{\partial x_{j}}\Delta y_{m}dx\right| + \left|\int_{\Omega}f(y_{m})\Delta y_{m}\,dx\right|+\left| \int_{\Omega}g \Delta y_{m}\,dx\right|.
\end{array}
$$
Using   H\"older's  inequality , the properties of the functions $a(\cdot)$ and $f(\cdot)$,   Young's  inequality, and the fact that $H^{1}(\Omega)\hookrightarrow L^{6}(\Omega)$  when $N=1,2$ or $3$; we have
\begin{equation}\label{A-eq3}
	\frac{1}{2}\frac{d}{dt}\|\nabla y_m\|^{2}+\frac{a_0}{2}\|\Delta y_m\|^{2} \leq \tilde{C}_{3} \|y_m\|^{2}+\tilde{C}_{3} \|\Delta y_m\|^{2}\|\nabla \Delta y_m\|^{2}+\tilde{C}_{3}\|g\|^{2}.
\end{equation}

{\bf Estimate III:} Taking $v=- \Delta y'_m(t)$ in \eqref{A3}, we have
$$
\begin{array}{l}
	\displaystyle \frac{1}{2}\|\nabla y'_{m}\|^{2} + \int_{\Omega}a(\nabla y_{m})\Delta y_{m}\Delta y'_{m}\,dx+\sum_{j=1}^{N}\int_{\Omega}Da(\nabla y_{m})\frac{\partial}{\partial x_{j}}(\nabla y_{m})\frac{\partial y_{m}}{\partial x_{j}}\Delta y'_{m}\,dx\\
	\noalign{\smallskip} \phantom{\frac{1}{2}||\nabla y'_{m}||}
\displaystyle -\int_{\Omega}f(y_{m})\Delta y'_{m}\,dx=-\int_{\Omega}g \Delta y'_{m}\,dx.
\end{array}
$$
Expressing  the second  term  on  the left  side  in an  appropriate  way, we have
$$
\begin{array}{l}
	\displaystyle \frac{1}{2} \|\nabla y'_{m}\|^{2} + \frac{1}{2} \frac{d}{dt}\left(\int_{\Omega}a(\nabla y_{m})|\Delta y_{m}|^{2}\,dx\right)-\frac{1}{2}\int_{\Omega}Da(\nabla y_{m})\cdot\nabla y'_{m}|\Delta y_{m}|^{2}\,dx\\
	\noalign{\smallskip} \phantom{\frac{1}{2}||\nabla y'_{m}||}
\displaystyle +\sum_{j=1}^{N}\int_{\Omega}Da(\nabla y_{m})\cdot\frac{\partial}{\partial x_{j}}(\nabla y_{m})\frac{\partial y_{m}}{\partial x_{j}}\Delta y'_{m}\,dx-\int_{\Omega}f(y_{m})\Delta y'_{m}\,dx=-\int_{\Omega}g \Delta y'_{m}\,dx.
\end{array}
$$
Passing  the last two terms  from the left  side  to the  right  side; and using  inequality's Holder, the properties of $a(\cdot)$ and $f(\cdot)$,  inequality's  Young, and the fact that $H^{1}(\Omega)\hookrightarrow  L^{6}(\Omega)$  when $N=1,2$ or $3$; we have
\begin{equation}\label{A-eq4}
\begin{array}{l}
	\displaystyle \frac{1}{2}\|\nabla y'_{m}\|^{2}+\frac{1}{2} \frac{d}{dt}\left(\int_{\Omega}a( \nabla y_m)|\Delta y_m|^{2}\, dx \right)\\
	\noalign{\smallskip} \phantom{\frac{1}{2}||\nabla y'_{m}||}
\displaystyle \leq \tilde{C}_{4} \left(\|\Delta y_m\|^{2} \|\nabla \Delta y_m\|^{2} + \|g\|^{2} + \|y_{m}\|^{2}\right)\\
	\noalign{\smallskip} \phantom{\frac{1}{2}||\nabla y'_{m}||AS}
\displaystyle + \epsilon\left( \|\Delta y_m\|^{2} + \|\Delta y'_m\|^{2}\right).
\end{array}
\end{equation}

{\bf Estimate IV:} Taking derivative with respect at time $t$ in the equation $\eqref{A3}_{1}$ and put $v=-\Delta y'_{m}(t)$, we have
$$
\begin{array}{l}
\displaystyle -\int_{\Omega}y''_{m}\Delta y'_{m}\,dx + \int_{\Omega} \nabla \cdot \Big[(Da(\nabla y_{m})\cdot \nabla y'_{m})\nabla y_{m} \Big] \Delta y'_{m}\, dx\\
	\noalign{\smallskip} \phantom{\frac{1}{2}||||AS}
\displaystyle +\int_{\Omega}\nabla \cdot \left( a(\nabla y_{m})\nabla y'_{m} \right)\Delta y'_{m}\,dx-\int_{\Omega}f'(y_{m})y'_{m}\Delta y'_{m}\,dx\\
	\noalign{\smallskip} \phantom{\frac{1}{2}||\nabla y'_{m}||AS}
\displaystyle =-\int_{\Omega}g'\Delta y'_{m}\,dx.
\end{array}
$$
Performing  some  operations on  the third term of the  left  side, we have
$$
\begin{array}{l}
	\displaystyle \frac{1}{2}\frac{d}{dt}\|\nabla y'_{m}\|^{2} + \int_{\Omega}a(\nabla y_{m})|\Delta y'_{m}|^{2}\,dx\\
	\noalign{\smallskip} \phantom{\frac{1}{2}||||AS}
\displaystyle = -\int_{\Omega} \nabla \cdot \Big[\left( Da(\nabla y_{m}) \cdot \nabla y'_{m}\right) \nabla y_{m} \Big] \Delta y'_{m}\,dx\\
	\noalign{\smallskip} \phantom{\frac{1}{2}||||ASAS}
\displaystyle -\sum_{j=1}^{N}\int_{\Omega}Da(\nabla y_{m})\cdot\frac{\partial}{\partial x_{j}}(\nabla y_{m})\frac{\partial y'_{m}}{\partial x_{j}}\Delta y'_{m}\,dx\\
	\noalign{\smallskip} \phantom{\frac{1}{2}||\nabla y'_{m}||S}
\displaystyle -\int_{\Omega}f'(y_{m})y'_{m}\Delta y'_{m}\,dx-\int_{\Omega}g'\Delta y'_{m}\,dx.
\end{array}
$$
Performing  some  operations  as in the Estimate III, we have
\begin{equation}\label{A-eq5}
\begin{array}{l}
	\displaystyle \frac{1}{2}\frac{d}{dt}\|\nabla y'_m\|^{2}+\frac{a_0}{2}\| \Delta y_m\|^{2}\\
	\noalign{\smallskip} \phantom{\frac{1}{2}||\nabla y'_{m}||AA}
\displaystyle \leq \tilde{C}_{5}\left( \|\Delta y'_m\|^{2} \|\nabla \Delta y_m\|^{2} + \| \nabla y'_m\|^{2} + \|g'\|^{2} \right).
\end{array}
\end{equation}
Notice that due to properties of the function $a(\cdot)$, we have that $\Delta y_{m}\in H^{1}_{0}(\Omega).$ Therefore, we can work with the following equation
\begin{equation}\label{A-eq6}
\left\{
\begin{array}{l}
\left(\Delta y'_{m},v\right)+\left(a(\nabla y_{m})\nabla \Delta y_{m}+\Delta (a(\nabla y_m))\nabla y_m, \nabla v \right)\\
	\noalign{\smallskip}\phantom{\int_{\Omega_1}\tilde{U}} 
\displaystyle 
 +\left(2\nabla\left(a(\nabla y_m)\right)\Delta y_m,\nabla v\right)+(\Delta\left(f(y_{m})\right),v)=\left(\Delta g,v\right), \,\, \forall \, v\in V,\\
	\noalign{\smallskip} 
\Delta y_{m}(0)=\Delta y_{0m} \to y_{0}\,\, \mbox{in}\,\, H^{1}_{0}(\Omega).
\end{array}
\right.
\end{equation}

Similarly as obtained in estimates I, II, III and IV, we will  obtain  new estimates  for $y_{m}$ and from \eqref{A-eq6}

{\bf Estimate V:} Taking $v=\Delta y_{m}(t)$ in \eqref{A-eq6}, we have
\begin{equation}\label{A-eq7}
\begin{array}{l}
	\displaystyle \frac{1}{2}\frac{d}{dt}\| \Delta y_m\|^{2} + \frac{a_0}{2}\|\nabla \Delta y_{m}\|^{2}\\
		\noalign{\smallskip} \phantom{\int_{\Omega_1}\tilde{U} \,dxG} 
\displaystyle \leq \tilde{C}_{6}\left( \|\nabla y_m\|^{2} + \|\Delta g\|^{2} + \|\Delta y_m\|^{4} \| \nabla \Delta y_m \|^{2}\right).
\end{array}
\end{equation}

{\bf Estimate VI:} Taking $v=\Delta^{2} y_{m}(t)$ in \eqref{A-eq6}, we have
\begin{equation}\label{A-eq8}
\begin{array}{l}
	\displaystyle \frac{1}{2}\frac{d}{dt}\|\nabla \Delta y_m\|^{2} + \frac{a_0}{2}\| \Delta^{2} y_{m}\|^{2}\\
		\noalign{\smallskip} \phantom{\int_{\Omega_1}\tilde{U} \,dxG} 
\displaystyle \leq \tilde{C}_{7}\left( \|\Delta y_m\|^{2} + \|\Delta g\|^{2} + \|\Delta y_m\|^{4} \|  \Delta^{2} y_m \|^{2}\right).
\end{array}
\end{equation}

{\bf Estimate VII:} Derivative with respect to $t$ in \eqref{A-eq6} and  taking $v=y''_{m}(t)$, we have
\begin{equation}\label{A-eq9}
\begin{array}{l}
	\displaystyle \|y''_m\|^{2} + \frac{1}{2}\frac{d}{dt}\left(\int_{\Omega}a(\nabla y_m)|\nabla y'_m|^{2}\, dx \right)\\
		\noalign{\smallskip} \phantom{\int_{\Omega_1}\tilde{U} \,dxG} 
\displaystyle \leq \tilde{C}_{8}\left(\|\Delta y'_m\|^{2} + \| g'\|^{2}+\|\Delta y'_m\|^{2} \| \nabla\Delta y_m \|^{2}\right).
\end{array}
\end{equation}
From all these estimates and taking $\displaystyle \epsilon=\frac{a_0}{16\tilde{C}_{5}}$ in \eqref{A-eq4}, we have
\begin{equation}\label{A-eq10}
\begin{array}{l}
	\displaystyle \frac{1}{2}\frac{d}{dt} \left( \tilde{C}_{9} \|\Delta y_m\|^{2} + \tilde{C}_{10} \|\nabla \Delta y_m\|^{2} + \tilde{C}_{11}\int_{\Omega}a(\nabla y_m)|\nabla y'_m|^{2}\, dx\right.\\
\noalign{\smallskip} \phantom{AA} \left.\displaystyle +3\tilde{C}_{5}\int_{\Omega}a(\nabla y_m)|\Delta y_m|^{2}\, dx + \|y_m\|^{2} + \|\nabla y_m\|^{2} + \|\nabla y'_m\|^{2}
 \right)\\
\noalign{\smallskip} \phantom{AA} \displaystyle +\tilde{C}_{9}\left(\frac{a_0}{8}-\tilde{C}_{6} \|\nabla\Delta y_m\|^{4} \right)\|\nabla\Delta y_m\|^{2} + \tilde{C}_{10}\left(\frac{a_0}{8}-\tilde{C}_{7}\|\nabla\Delta y_m\|^{4}\right) \|\Delta^{2}y_m\|^{2}\\
\noalign{\smallskip}\phantom{AA} \displaystyle + \left(\frac{a_0}{8}-\tilde{C}_{12} \|\nabla\Delta y_m\|^{4} \right) \|\Delta y_m\|^{2}+\left(\frac{a_0}{8}-\tilde{C}_{13} \|\nabla\Delta y_m\|^{4}\right) \|\Delta y'_m\|^{2}\\
\noalign{\smallskip}\phantom{AA}
\displaystyle +\tilde{C}_{11}\|y''_m\|^{2}\displaystyle +\frac{a_0}{4}\|\nabla y_m\|^{2}+\frac{\tilde{C}_{5}}{2}\|\nabla y'_m\|^{2} \leq \tilde{C}_{13} \|\Delta g\|^{2} + \tilde{C}_{14} \|g'\|^{2}\\
\noalign{\smallskip}\phantom{AA}
\displaystyle + \tilde{C}_{15} \left(\|y_0\|^{2}+\|g\|^{2}_{L^{2}(Q)}\right),
\end{array}
\end{equation}
where,
$$\begin{array}{l}
\displaystyle \tilde{C}_{9}=\frac{a_0}{16\tilde{C}_{6}},\,\, \tilde{C}_{10}=\frac{a_0}{16\tilde{C}_{7}},\,\, \tilde{C}_{11}=\frac{a_0}{16\tilde{C}_{8}},\,\, \tilde{C}_{12}=\frac{2(3\tilde{C}_{4}\tilde{C}_{5}+\tilde{C}_{3})^{2}}{a_0},\,\,\tilde{C}_{13}=\frac{2(\tilde{C}_{5}+\frac{a_0}{16})^{2}}{a_0}, \\
\noalign{\smallskip}
\displaystyle \tilde{C}_{14}=\left(\tilde{C}_{1}+\tilde{C}_{3}+\tilde{C}_{4}+\tilde{C}_{5}+3\tilde{C}_{4}\tilde{C}_{5}\right)C(\Omega)+\frac{a_0}{8},\,\, \tilde{C}_{15}=\tilde{C}_{1}+\tilde{C}_{3},\,\,\tilde{C}_{16}=\frac{a_0}{16}+\tilde{C}_{5}.
\end{array}
$$
Also, from \eqref{A3} taking $v=-\Delta y_{m}(t)$. We have
$$\|\nabla y'_{m}(0)\|^{2}\leq \tilde{C}_{17}\left( \|\nabla\Delta y_{0}\|^{2} + \|\nabla g(0)\|^{2}\right).$$
There exist $\epsilon_{0}>0$ such that for 
$$\|y_{0}\|_{H^{3}(\Omega)} + \|g\|_{X_{0}}< \epsilon_{0},$$
we have
\begin{equation}\label{A-eq11}
\begin{array}{l}
	\displaystyle \|\nabla\Delta y_{0}\|^{4}< \frac{a_0}{8\left(\tilde{C}_{6}+\tilde{C}_{7}+\tilde{C}_{12}+\tilde{C}_{13}\right)},
\end{array}
\end{equation}
and 
\begin{equation}
\left\{\begin{array}{l}
	\displaystyle (\tilde{C}_{9}+3\tilde{C}_{5}a_{1}) \|\Delta y_{0}\|^{2} + \tilde{C}_{10}\|\nabla\Delta y_{0}\|^{2} + (1+a_{1}\tilde{C}_{11}\tilde{C}_{16})\left(\|\nabla\Delta y_{0}\|^{2}+\|\nabla g(0)\|^{2}\right) \\
\noalign{\smallskip}	 
\displaystyle +\|y_{0}\|^{2}+\|\nabla y_0\|^{2}+2\tilde{C}_{15}T \left(\|y_0\|^{2}+\|g\|^{2}\right)+2\tilde{C}_{13}\|\Delta g\|^{2}\\
\noalign{\smallskip} 
\displaystyle + 2\tilde{C}_{14}\|g'\|^{2}_{L^{2}(Q)} < \frac{\tilde{C}_{10}\ a^{\frac{1}{2}}_{0}}{4\left(\tilde{C}_{6}+\tilde{C}_{7}+\tilde{C}_{12}+\tilde{C}_{13}\right)^{\frac{1}{2}}}.
\end{array}\right.
\end{equation}
Thanks to these inequalities we can prove by a contradiction argument that,  the following estimate holds
$$
\|\nabla\Delta y_{m}(t)\|^{4}< \frac{a_0}{8\left(\tilde{C}_{6}+\tilde{C}_{7}+\tilde{C}_{12}+\tilde{C}_{13}\right)}, \,\,\forall\, t\geq 0.
$$
Now, we can obtain from estimate \eqref{A-eq10}
$$
\left\{\begin{array}{l}
(y_m)\,\, \mbox{is bounded in }\,\, L^{\infty}(0,T;H^{3}(\Omega))\cap L^{2}(0,T;H^{4}(\Omega)),\\
\noalign{\smallskip} 
(y'_{m})\,\, \mbox{is bounded in }\,\, L^{\infty}(0,T;H^{1}_{0}(\Omega))\cap L^{2}(0,T;H^{2}(\Omega)),\\
\noalign{\smallskip} 
(y''_{m})\,\, \mbox{is bounded in }\,\, L^{2}(0,T;L^{2}(\Omega)).
\end{array}\right.
$$
All these uniform bounds allow to take limits in \eqref{A3} (at least for a subsequence) as $m\to \infty$. Indeed, the unique delicate point is the a.e convergence of $a(\nabla y_m)$. But this is a consequence of the fact that the sequence $(y_m)$ is pre-compact in $L^{2}(0,T;H^{2}(\Omega))$ and $a \in C^{3}(\mathbb{R}^{N})$.

The uniqueness of the strong solution to \eqref{A1} can be proved by argument standards (to see \cite{Rincon}).

\end{proof}
A consequence of Lemma \ref{well-p1} is the following result, which guarantees that system \eqref{EC1} is well-possessedness:

\begin{lemma}\label{well-p2}
There exists $r>0$ such that  for each $y_{0},\, \hat{y}_{0}\in H^{3}(\Omega)\cap H^{1}_{0}(\Omega)$ with $\Delta y_0, \Delta \hat{y}_0 \in H_0^1(\Omega)$ and $\xi,\, u\in X_{0}$  with   $\mathrm{supp}\ u\subset \omega\times(0,T)$ satisfying
$$
\|y_{0}+\tau \hat{y}_{0}\|_{H^{3}(\Omega)} + \|\xi\|_{X_{0}} + \|u\|_{X_0}\leq r,
$$
the problem \eqref{EC1} has a unique solution $y$ in $X_{1}$ satisfying
\begin{equation}\label{A4}
	\|y\|_{X_{1}}\leq C \left( \|y_{0}+\tau \hat{y}_{0}\|_{H^{3}(\Omega)}+ \|\xi\|_{X_{0}} + \|u\|_{X_{0}}\right).
\end{equation}
\end{lemma}

\begin{proof}
Notice that the estimate \eqref{A2}, we have that \eqref{A4} holds.
\end{proof}


\section{Well-possessedness  for  the system \eqref{eq-char} }\label{Apex_2}
\begin{lemma}\label{Lemma-eq-char}
There exists $\tilde{r}>0$ such that  for each  $\xi,\, u\in X_{0}$  with   $\mathrm{supp}\  u\subset \omega\times(0,T)$ satisfying
$$
\|\xi\|_{X_{0}}+\|u\|_{X_0}\leq \tilde{r},
$$
the problem \eqref{char-e} has a unique solution $(y,h)$ in $X_{1}\times X_{0}$ satisfying
\begin{equation}\label{A5}
	\|y\|_{X_{1}}+\|h\|_{X_{0}}\leq C \left( \|\xi\|_{X_{0}} + \|u\|_{X_{0}}\right).
\end{equation}
\end{lemma}

\begin{proof}
We  employ Faedo-Galerkin method with  the Hilbert basis from $H^{1}_{0}(\Omega)$, given by  eigenvectors $(w_{j})$ of the spectral problem $\left((w_{j},v)\right)=\lambda_{j} (w_{j},v),$ for all $v\in V=H^{2}(\Omega)\cap H^{1}_{0}(\Omega)$ and $j=1,2,3,...$ We  represent  by $V_{m}$ the subspace of $V$ generated by vectors $\{w_1,w_2,...,w_m\}.$ We propose the following approximate problem:
\begin{equation}\label{A5-eq1}
\left\{
\begin{array}{l}
	\left(y'_{m},v\right)+\left(a(\nabla y_{m})\nabla y_{m},\nabla v\right)+(f(y_{m}),v)=\left(g,v\right), \,\, \forall \, v\in V_m,\\
	\noalign{\smallskip} 
\left( -h'_{m},w\right)+\left(Da(\nabla y_{m})\left(\nabla y_{m}\cdot\nabla h_{m}\right)+a(\nabla y_{m})\nabla h_{m}, \nabla w \right)\\
	 \noalign{\smallskip}\phantom{\int_G-jjkjk}
+\left(f'(y_{m})h_{m},w \right)=\left(y_{m}\chi_{\mathcal{O}},w \right), \,\,\forall\, w\in V_{m},\\
	\noalign{\smallskip} 
y_{m}(0)=y_{0m} \to 0\,\, \mbox{in}\,\, H^{3}(\Omega)\cap H^{1}_{0}(\Omega),\\
	\noalign{\smallskip} 
h_{m}(T)=h_{0m}\to 0\,\, \mbox{in}\,\, H^{1}_{0}(\Omega).\\
\end{array}
\right.
\end{equation}
From the proof of Lemma \ref{well-p1}, we have that

\begin{equation}\label{A5-eq2}
\begin{array}{l}
	\|\Delta y_{m}(t)\|^{2}+\|\nabla\Delta y_{m}(t)\|^{2}+\|\nabla y'_{m}(t)\|^{2}\\
\noalign{\smallskip}\phantom{\int_G-jjkjk}
+\|y_{m}\|^{2}_{L^{\infty}(0,T;L^{2}(\Omega))} \leq \hat{C}_{1}\left(\|\xi\|^{2}_{X_{0}} + \|u\|^{2}_{X_{0}} \right).
\end{array}
\end{equation}

The existence and uniqueness of (local in time)  solution to the  $\eqref{A5-eq1}_{2},\, \eqref{A5-eq1}_{3}$ and $\eqref{A5-eq1}_{5}$  is ensure  by classical ODE theory. In order to extend the solution $h_{m}$ for $[0,T]$,  we will  obtain  estimates for $h_{m}$

{\bf Estimate I:} Taking $w=h_{m}(t)$  in \ref{A5-eq1}, we have

$$
\begin{array}{l}
	\displaystyle -\frac{d}{dt}\|h_{m}\|^{2}+\int_{\Omega}a(\nabla y_{m})|\nabla h_{m}|^{2}\,dx+\int_{\Omega}Da(\nabla y_{m})(\nabla y_{m}\cdot \nabla h_{m})\cdot\nabla h_{m}\,dx\\
\noalign{\smallskip}\phantom{\int_G-jjkjk}
\displaystyle =\int_{\Omega}y_{m}\chi_{\mathcal{O}}h_{m}\,dx.
\end{array}
$$
In the sequel, the symbol $\hat{C}_{k}$ is a constant that only depend of $a_0,\, a_1,\, M$ and $\Omega$ for $k=1,...,10.$
Passing  the last two terms  from left side to the right  side  and using  some operations along  with  the H\"older's inequality  and Young's inequality, we have
\begin{equation}\label{A5-eq3}
\begin{array}{l}
	\displaystyle -\frac{d}{dt}\|h_{m}\|^{2}+\frac{a_{0}}{2}\|\nabla h_{m}\|^{2}\leq \hat{C}_{2}\left( \|\nabla \Delta y_{m}\|^{2} \|\nabla h_{m}\|^{2}+\|h_{m}\|^{2}+\|y_{m}\|^{2} \right).
\end{array}
\end{equation}
Notice that in the proof of Lemma \ref{well-p1}, when $\xi$ and $u$  are  small sufficiently, we have that 
$$
\hat{C}_{2}\|\Delta y_{m}\|^{2}\leq \frac{a_{0}}{4},
$$
using this and by \ref{A5-eq3}, we have
\begin{equation}\label{A5-eq4}
\begin{array}{l}
	\displaystyle -\frac{d}{dt}\|h_{m}\|^{2}+\frac{3a_{0}}{4}\|\nabla h_{m}\|^{2}\leq \hat{C}_{2}\left(\|h_{m}\|^{2}+\|y_{m}\|^{2} \right),
\end{array}
\end{equation}
and by Gronwall's inequality and using \ref{A5-eq2},  we have
\begin{equation}\label{A5-eq5}
	\|h_{m}\|^{2}+\|\nabla h_{m}\|^{2}_{L^{2}(Q)}\leq \hat{C}_{3}\left(\|\xi\|^{2}_{X_{0}}+\|u\|^{2}_{X_{0}} \right).
\end{equation}

Notice  that the term of side right in \ref{A5-eq5}  don't dependent of $m$ and due it, we can  extend  the  solution $(h_{m})$  to the interval $[0,T]$.

{\bf Estimate II:} Taking $w=-\Delta h_{m}$  in \eqref{A5-eq2}, we have
$$
\begin{array}{l}
	\displaystyle -\frac{d}{dt}\|h_{m}\|^{2}+\int_{\Omega}\nabla \cdot \Big( Da(\nabla y_{m})\left(\nabla y_{m}\cdot \nabla h_{m}\right) \Big) \Delta h_{m}\,dx\\
\noalign{\smallskip}\phantom{\int_G-jjk}

\displaystyle +\int_{\Omega}a(\nabla y_{m})|\Delta h_{m}|^{2}\,dx+\int_{\Omega}Da(\nabla y_{m})\Delta y_{m}\nabla y_{m}\Delta h_{m}\,dx\\
\noalign{\smallskip}\phantom{\int_G-jjk}

\displaystyle -\int_{\Omega}f'(y_{m})h_{m}\Delta h_{m}\,dx=-\int_{\Omega}y_{m}\chi_{\mathcal{O}}\Delta h_{m}\,dx.
\end{array}
$$
Passing  the last three terms  from left side to the right  side  and using  some operations along  with  the H\"older's  inequality  and Young's inequality, we have
\begin{equation}\label{A5-eq6}
\begin{array}{l}
\displaystyle -\frac{d}{dt}\|h_{m}\|^{2} + a_{0}\|\Delta h_{m}\|^{2}\leq \\
\noalign{\smallskip}\phantom{\int_G-jjk}

\displaystyle \hat{C}_{4}\left(\|\nabla \Delta y_{m}\|^{2} \|\Delta y_{m}\|^{2} + \|\nabla y_{m}\|^{2} \right)\|\Delta h_{m}\|^{2}\\
\noalign{\smallskip}\phantom{\int_G-jjkDDDss}

\displaystyle +\hat{C}_{4}\|\Delta y_{m}\|^{2} \|\nabla \Delta y_{m}\|^{2} \|\nabla h_{m}\|^{2} +\frac{a_{0}}{4}\|\Delta h_{m}\|^{2}.
\end{array}
\end{equation}
Notice that in the proof of Lemma \ref{well-p1}, when $\xi$ and $u$  are  small sufficiently, we have 
$$
\hat{C}_{5}\|\Delta y_{m}\|^{2} \|\nabla \Delta y_{m}\|^{2}\leq \frac{a_{0}}{4}
$$
and 
$$
\hat{C}_{5}\|\nabla \Delta y_{m}\|^{2} \leq \frac{a_{0}}{4}.
$$
By last two inequalities and  the inequality \eqref{A5-eq5}  in \eqref{A5-eq6}, we have
$$
\begin{array}{l}
	\displaystyle -\frac{d}{dt}\|\nabla h_{m}\|^{2}+\frac{a_{0}}{4}\|\Delta y_{m}\|^{2} \leq\displaystyle \hat{C}_{6} \left( \|u\|^{2}_{X_{0}}+ \|\xi\|^{2}_{X_{0}}+ \|\nabla h_{m}\|^{2} \right).
\end{array}
$$
Using Gronwall's inequality, we have

\begin{equation}\label{A5-eq7}
	\displaystyle \|\nabla h_{m}\|^{2} + \frac{a_{0}}{4}\|\Delta h_{m}\|^{2}_{L^{2}(Q)} \leq \hat{C}_{7}\left(\|u\|^{2}_{X_{0}}+\|\xi\|^{2}_{X_{0}} \right).
\end{equation}

{\bf Estimate III:} Taking $w=h'_{m}$  in \eqref{A5-eq2}, we have
$$
\begin{array}{l}
	\displaystyle \|h'_{m}\|^{2} + \int_{\Omega}\nabla \cdot \Big( Da(\nabla y_{m})\left(\nabla y\cdot\nabla h_{m}\right)+a(\nabla y_{m})\nabla h_{m} \Big) h'_{m}\,dx\\
\noalign{\smallskip}\phantom{\int_G-jjkDDDssDF}

\displaystyle-\int_{\Omega}f'(y_{m})h_{m}h'_{m}\,dx=-\int_{\Omega}y_{m}\chi_{\mathcal{O}}h'_{m}\,dx.
\end{array}
$$
Passing  the last two terms  from left side to the right  side  and using  some operations along  with  the H\"older's inequality and Young's inequality , we have
$$
\begin{array}{l}
	\displaystyle \|h'_{m}\|^{2}\leq \frac{\|h'_{m}\|^{2}}{2} + \hat{C}_{8}\left(\|\Delta h_{m}\|^{2} + \|\nabla \Delta y_{m}\|^{2} \|\Delta h_{m}\|^{2}\right).
\end{array}
$$
From \eqref{A5-eq2} and integrating in $(0,T)$ 
$$
\displaystyle \|h'_{m}\|^{2}\leq \hat{C}_{9} \|\Delta h_{m}\|^{2},
$$
and using \eqref{A5-eq7}, we have
\begin{equation}\label{A5-eq9}
\begin{array}{l}
	\displaystyle \|h'_{m}\|^{2}_{L^{2}(Q)} \leq \hat{C}_{10}\left(\|u\|^{2}_{X_{0}}+\|\xi\|^{2}_{X_{0}} \right).
\end{array}
\end{equation}
Now, we can obtain from estimate \eqref{A5-eq2}, \eqref{A5-eq5},\eqref{A5-eq7} and \eqref{A5-eq9} the following results
$$
\left\{\begin{array}{l}
(y_m)\,\, \mbox{is bounded in }\,\, L^{\infty}(0,T;H^{3}(\Omega))\cap L^{2}(0,T;H^{4}(\Omega)),\\
\noalign{\smallskip} 
(y'_{m})\,\, \mbox{is bounded in }\,\, L^{\infty}(0,T;H^{1}_{0}(\Omega))\cap L^{2}(0,T;H^{2}(\Omega)),\\
\noalign{\smallskip} 
(y''_{m})\,\, \mbox{is bounded in }\,\, L^{2}(0,T;L^{2}(\Omega)),\\
\noalign{\smallskip} 
(h_m)\,\, \mbox{is bounded in }\,\, L^{\infty}(0,T;H^{1}_{0}(\Omega))\cap L^{2}(0,T;H^{2}(\Omega)),\\
\noalign{\smallskip} 
(h'_{m})\,\, \mbox{is bounded in }\,\, L^{2}(0,T;L^{2}(\Omega)).\\
\end{array}\right.
$$
All these uniform bounds allow to take limits in \eqref{A5-eq1} (at least for a subsequence) as $m\to \infty$. Indeed, the unique delicate point is the a.e convergence of $a(\nabla y_m)$. But this is a consequence of the fact that the sequence $(y_m)$ is pre-compact in $L^{2}(0,T;H^{2}(\Omega))$ and $a \in C^{3}(\mathbb{R}^{N})$.

The uniqueness of the strong solution to \eqref{A1} can be proved by argument standards (to see \cite{Rincon}).

\end{proof}


\section{Proof of Proposition \ref{char-e} }\label{Apex_3}

For  any $\hat{y}_{0}\in H^{3}(\Omega)\cap H^{1}_{0}(\Omega)$  satisfying $\|\hat{y}_{0}\|_{H^{3}(\Omega)}=1$ and $|\tau|<\tau_{0}$, let us denote by  $y_{\tau}$  the solution of equation \eqref{EC1}  associated to $\tau$  and $u$. For to prove \eqref{eq-def-ins}, first we prove

\begin{equation}\label{AP-eq1}
    \left. \frac{\partial \phi}{\partial \tau}(y_{\tau}) \right|_{\tau =0}=\int_{\Omega}h(0)\hat{y}_{0}\,dx, \,\,\,\,\forall\, \hat{y}_{0}\in H^{3}(\Omega)\cap H^{1}_{0}(\Omega)\,\,\,\mbox{with}\,\,\,\|\hat{y}_{0}\|_{H^{3}(\Omega)}=1,
\end{equation}

and 

\begin{equation}\label{AP-eq2}
y_{\tau}\to y\,\,\mbox{in}\,\,L^{2}(Q)\,\,\mbox{as}\,\,\tau \to 0.
\end{equation}

Let us denote by $w_{\tau}=y_{\tau}-y$. Then, $w_{\tau}$ satisfy
\begin{equation}\label{AP-eq3}
\left\{\begin{array}{lll}
w_{\tau,t}-\nabla \cdot \Big(a(\nabla y_{\tau})\nabla y_{\tau}-a(\nabla y)\nabla y \Big) + f(y_{\tau})-f(y)=0 & \mbox{in} & Q,\\
\noalign{\smallskip} 

w_{\tau}=0 &\mbox{on} & \Sigma,\\
\noalign{\smallskip} 

w_{\tau}(0)=\tau \hat{y}_{0} &\mbox{in} &\Omega,
\end{array}\right.
\end{equation}
multiplying on both  sides in $\eqref{AP-eq3}_{1}$ by $w_{\tau}$  and  integrating in $\Omega$, we have

$$
\begin{array}{l}
\displaystyle \frac{1}{2}\frac{d}{dt} \left( \int_{\Omega}|w_{\tau}|^{2}\,dx \right) + \int_{\Omega}a(\nabla y_{\tau})|\nabla w_{\tau}|^{2}\,dx\\
\noalign{\smallskip} 

\displaystyle +\int_{\Omega} \Big(a(\nabla y_{\tau})-a(\nabla y) \Big) \nabla y \cdot \nabla w_{\tau}\,dx + \int_{\Omega} \Big(f(y_{\tau})-f(y) \Big) w_{\tau}\,dx=0.
\end{array}
$$
Using the properties of $a(\cdot)$, $f(\cdot)$, Poincare's inequality and Young's inequality, we have
$$
\begin{array}{l}
	\displaystyle \frac{1}{2}\frac{d}{dt}\|w_{\tau}\|^{2} + \frac{a_{0}}{2}\|\nabla w_{\tau}\|^{2}\leq C \|w_{\tau}\|^{2}.
\end{array}
$$
By  Gronwall's inequality, we  obtain that
$$
\|w_{\tau}\|^{2}\leq C\|w_{\tau}(0)\|^{2}.
$$
Due to $w_{\tau}(0)\to 0$  in $L^{2}(\Omega)$  (as $\tau \to 0$), we find that
$$
w_{\tau}\to 0\,\,\,\mbox{in}\,\,\, L^{\infty}(0,T;L^{2}(\Omega)).
$$
Therefore,
$$
y_{\tau}\to y\,\,\,\mbox{in}\,\,\,L^{\infty}(0,T;L^{2}(\Omega)),\,\,\,\mbox{as}\,\,\tau\to 0.
$$
This yield \eqref{AP-eq2}

In order to prove \eqref{AP-eq1}, we need to prove that
$$
\frac{y_{\tau}-y}{\tau}\to p\,\,\,\mbox{in}\,\,\,L^{2}(Q)\,\,\,\mbox{as}\,\,\,\tau \to 0,
$$
where $p$  satisfies the following  equation

\begin{equation}\label{AP-eq4}
\left\{\begin{array}{lll}
p_{t}-\nabla \cdot \Big(a(\nabla y)\nabla p+(Da(\nabla y)\cdot \nabla p)\nabla y \Big) + f'(y)p = 0 & \mbox{in} &  Q,\\
\noalign{\smallskip} 

p=0  &\mbox{on} &\Sigma,\\
\noalign{\smallskip} 

p(0)=\hat{y}_{0} & \mbox{in} &\Omega.
\end{array}\right.
\end{equation}

We denote by $\displaystyle p_{\tau}=\frac{y_{\tau}-y}{\tau}-p$,  then it is easy to check that $p_{\tau}$ satisfies

\begin{equation}\label{AP-eq5}
\left\{\begin{array}{lll}
\displaystyle p_{\tau,t}-\nabla \cdot \Big(a(\nabla y_{\tau})\nabla p_{\tau} \Big) + \nabla \cdot \Big((a(\nabla y))-a(\nabla y_{\tau})\nabla p \Big) &  &  \\
\noalign{\smallskip} 
\phantom{\int_G}
\displaystyle -\nabla \cdot\left[\left(\frac{a(\nabla y_{\tau})-a(\nabla y)}{\tau}-Da(\nabla y)\cdot \nabla p \right)\nabla y \right] &\mbox{in}  & Q, \\
\noalign{\smallskip} 
\phantom{\int_G}
\displaystyle +\left(\frac{f(y_{\tau})-f(y)}{\tau}-f'(y)p \right)=0 &  &\\
\noalign{\smallskip} 

p_{\tau}=0 &\mbox{on} & \Sigma,\\
\noalign{\smallskip} 

p_{\tau}(0)=0 &\mbox{in}  &\Omega. 
\end{array}\right.
\end{equation}

Multiplying both  sides  of the  first  equation  of \eqref{AP-eq5}  by $p_{\tau}$ and integrating  it  in $\Omega$  we  obtain  that

\begin{equation}\label{AP-eq6}
\begin{array}{l}

\displaystyle \frac{1}{2}\frac{d}{dt}\int_{\Omega}|p_{\tau}|^{2}\,dx+\int_{\Omega}a(\nabla y_{\tau})|\nabla p_{\tau}|^{2}\,dx\\
\noalign{\smallskip} 

\phantom{\int_G-jjk}
\displaystyle =\int_{\Omega} \Big(a(\nabla y)-a(\nabla y_{\tau})\Big)\nabla p \nabla p_{\tau}\,dx\\
\noalign{\smallskip} 
\phantom{\int_G-jjk}

\displaystyle -\int_{\Omega}\left(\frac{a(\nabla y_{\tau})-a(\nabla y)}{\tau}-Da(\nabla y)\cdot \nabla p  \right)\nabla y  \nabla p_{\tau}\,dx\\
\noalign{\smallskip} 
\phantom{\int_G-jjk}

\displaystyle -\int_{\Omega}\left(\frac{f(y_{\tau})-f(y)}{\tau}-f'(y)p \right)p_{\tau}\,dx\\
\noalign{\smallskip} 
\phantom{\int_G-jjk}

= J_{1}+J_{2}+J_{3}.
\end{array}
\end{equation}

Using  the properties of $a(\cdot)$  and   H\"older's  inequality , we have

\begin{equation}\label{AP-eq7}
\begin{array}{l}
\displaystyle |J_{1}|\leq \epsilon \int_{\Omega}|\nabla p_{\tau}|^{2}\,dx+C\int_{\Omega}|w_{\tau}|^{2}\,dx.
\end{array}
\end{equation}
Similarly, for some $\lambda_{*}\in[0,1]$ \,we have

\begin{equation}\label{AP-eq7-2}
\begin{array}{l}
\displaystyle \displaystyle |J_{2}|\leq \left|\int_{\Omega}\left(Da(\lambda_{*}\nabla y_{\tau}+(1-\lambda_{*})\nabla y)\cdot (\nabla \frac{w_{\tau}}{\tau})-Da(\nabla y)\cdot \nabla p \right)\nabla  y\nabla p_{\tau}\,dx \right|\\
\noalign{\smallskip} 
\phantom{\int_G-jjk}
\displaystyle \leq  \int_{\Omega}|Da(\lambda_{*}\nabla y_{\tau}+(1-\lambda_{*})\nabla y)-Da(\nabla y)|\cdot | \nabla p |\cdot |\nabla  y|\cdot|\nabla p_{\tau} |\,dx\\

\noalign{\smallskip} 
\phantom{\int_G-jjk}
+ \displaystyle \int_{\Omega}|Da(\lambda_{*}\nabla y_{\tau}+(1-\lambda_{*})\nabla y )|\cdot |\nabla  y|\cdot|\nabla p_{\tau} |^{2}\,dx\\

\noalign{\smallskip} 
\phantom{\int_G-jjk}
\displaystyle \leq  C\int_{\Omega}|\nabla w_{\tau}|\cdot | \nabla p |\cdot |\nabla  y|\cdot|\nabla p_{\tau} |\,dx
+ \displaystyle C\int_{\Omega}|\nabla  y|\cdot|\nabla p_{\tau} |^{2}\,dx.
\end{array}
\end{equation}
For  $\xi$  and $u$ sufficiently  small and from Lemma \ref{Lemma-eq-char}, we have
$$
\begin{array}{l}
	\|\nabla y\|_{C(0,T;L^{\infty}(\Omega))}\leq C\left(\|\xi\|_{X_{0}}+\|u\|_{X_{0}} \right),
\end{array}
$$
and 
$$
\|\nabla p\|_{C(0,T;L^{\infty}(\Omega))} \leq C\left( \|\xi\|_{X_{0}} + \|u\|_{X_{0}} + \|\hat{y}_{0}\|_{H^{3}(\Omega)} \right).
$$
Using   H\"older's  inequality in \eqref{AP-eq7-2} and  this last inequality,   we have,
\begin{equation}\label{AP-eq8}
|J_{2}|\leq \epsilon \int_{\Omega}|\nabla p_{\tau}|^{2}\,dx+C\int_{\Omega}|\nabla w_{\tau}|^{2}\,dx.
\end{equation}
Also, 
$$
\begin{array}{l}
\displaystyle |J_{3}|= \left| \int_{\Omega}\left( \frac{f(y_{\tau})-f(y)}{\tau}-f'(y)p\right)p_{\tau}\,dx \right|\\

\noalign{\smallskip} 
\phantom{|J_{3}|}
 \displaystyle = \left|\int_{\Omega}\left[ f'(\tilde{\lambda}_{*}y_{\tau}-(1-\tilde{\lambda}_{*}))p_{\tau}-\left(f'(\tilde{\lambda}_{*} y_{\tau}+(1-\tilde{\lambda}_{*})y)-f'(y)\right)p  \right]p_{\tau}\,dx \right|,\\
\end{array}
$$
for  some  $\tilde{\lambda}_{*}\in[0,1]$.

Now, using the fact that $f''$ is bounded and by Young's inequality,

$$
\begin{array}{l}

\displaystyle  |J_{3}|\leq C\int_{\Omega}|p_{\tau}|^{2}\,dx+\int_{\Omega}|w_{\tau}|\cdot|p|\cdot|p_{\tau}|\,dx\\
\noalign{\smallskip} 
\phantom{|J_{3}|}

\displaystyle \leq C\int_{\Omega}|p_{\tau}|^{2}\,dx+\int_{\Omega}|w_{\tau}|^{2}\cdot|p|^{2}\,dx\\
\noalign{\smallskip} 
\phantom{|J_{3}|}

\displaystyle \leq C\int_{\Omega}|p_{\tau}|^{2}\,dx+\|p\|^{2}_{C(0,T;L^{\infty}(\Omega))}\int_{\Omega}|w_{\tau}|^{2}\,dx,\\

\end{array}
$$
and we have,
\begin{equation}\label{AP-eq9}
|J_{3}|\leq C\int_{\Omega}|p_{\tau}|^{2}\,dx+C\int_{\Omega}|w_{\tau}|^{2}\,dx.
\end{equation}
From \eqref{AP-eq7}, \eqref{AP-eq8} and \eqref{AP-eq9}, we have
$$
\begin{array}{l}
\displaystyle \frac{1}{2}\frac{d}{dt}\int_{\Omega}|p_{\tau}|^{2}\,dx+\frac{a_{0}}{2}\int_{\Omega}|\nabla p_{\tau}|^{2}\,dx\\
\noalign{\smallskip} 
\phantom{|J_{3}|DDDDD}
\displaystyle \leq C\int_{\Omega}|w_{\tau}|^{2}\,dx+C\int_{\Omega}|p_{\tau}|^{2}\,dx.\\
\end{array}
$$
Using  Gronwall's inequality, we have
$$
\int_{\Omega}|p_{\tau}|^{2}\,dx\leq C\int_{\Omega}|w_{\tau}|^{2}\,dx.
$$
From last inequality, since
$$w_{\tau}\to 0\,\,\,\,\mbox{as}\,\,\,\,\,\,\tau \to 0,$$
we have that,
$$
p_{\tau}\to 0\,\,\,\,\,\mbox{as} \,\,\,\, \tau\to 0.
$$
Therefore,
$$
\frac{y_{\tau}-y}{\tau}\to p\,\,\,\,\,\mbox{in}\,\,\,\, L^{2}(Q)\,\,\,\,\mbox{as}\,\,\, \tau\to 0.
$$
We  know that,
$$
 \left.\frac{\partial \Phi(y_{\tau})}{\partial \tau} \right|_{\tau=0}=\frac{1}{2}\lim_{\tau \to 0}\iint_{\mathcal{O}\times (0,T)}(y_{\tau}+y)\left(\frac{y_{\tau}-y}{\tau} \right) dxdt,
$$
where $y$  is the solution of \eqref{eq-char}, then
\begin{equation}\label{AP-eq10}
\left.\frac{\partial \Phi(y_{\tau})}{\partial \tau} \right|_{\tau=0}=\iint_{\mathcal{O}\times(0,T)}py\,dxdt,
\end{equation}
with $p$  is the solution in \eqref{AP-eq4}.\\

Now, multiplying in \eqref{AP-eq4}  by $h$  and integrate  in $Q$, we have
\begin{equation}\label{AP-eq11}
\iint_{\mathcal{O}\times (0,T)}py\,dxdt=\int_{\Omega}h(0)\hat{y}\,dx, \,\,\,\forall\,\hat{y}_{0}\in H^{3}(\Omega)\,\,\,\mbox{with}\,\,\,\|\hat{y}_{0}\|_{H^{3}(\Omega)}=1,
\end{equation}
from \eqref{AP-eq10} and \eqref{AP-eq11}, we have
$$
\left.\frac{\partial \Phi(y_{\tau})}{\partial \tau} \right|_{\tau=0}=0\,\,\,\,\mbox{if  and only if}\,\,\,\,\,\,h(0)=0.
$$
This complete the proof of Proposition \ref{char-e}.

\end{appendices}



\end{document}